\documentclass[10pt]{article}
\usepackage{amsmath, amssymb, amsfonts, amsthm, epsfig, fdsymbol, float, graphicx, sectsty}
\usepackage[all]{xy}
\title{Local permutation stability}
\author{Henry Bradford}

\newtheorem{thm}{Theorem}[section]
\newtheorem{lem}[thm]{Lemma}
\newtheorem{propn}[thm]{Proposition}
\newtheorem{coroll}[thm]{Corollary}
\newtheorem{defn}[thm]{Definition}
\newtheorem{ex}[thm]{Example}

\newtheorem{conj}[thm]{Conjecture}

\newtheorem{rmrk}[thm]{Remark}
\newtheorem{qu}[thm]{Question}
\newtheorem{prob}[thm]{Problem}

\DeclareMathOperator{\Alt}{Alt}

\DeclareMathOperator{\BS}{BS}

\DeclareMathOperator{\diam}{diam}
\DeclareMathOperator{\FAlt}{FAlt}
\DeclareMathOperator{\FSym}{FSym}
\DeclareMathOperator{\gen}{gen}
\DeclareMathOperator{\GL}{GL}
\DeclareMathOperator{\Homeo}{Homeo}
\DeclareMathOperator{\id}{id}
\DeclareMathOperator{\im}{im}
\DeclareMathOperator{\IRS}{IRS}
\DeclareMathOperator{\Prob}{Prob}

\DeclareMathOperator{\SL}{SL}

\DeclareMathOperator{\Stab}{Stab}
\DeclareMathOperator{\stat}{stat}
\DeclareMathOperator{\Sub}{Sub}

\DeclareMathOperator{\Sym}{Sym}

\begin{document}

\maketitle

\begin{abstract}
We introduce a notion of ``local stability in permutations'' 
for finitely generated groups. 
If a group is sofic and locally stable in our sense, 
then it is also locally embeddable into finite groups (LEF). 
Our notion is weaker than the ``permutation stability'' introduced by 
Glebsky-Rivera and Arzhantseva-Paunescu, 
which allows one to upgrade soficity to residual finiteness. 
We prove a necessary and sufficient condition 
for a finitely generated amenable group to be locally permutation stable, 
in terms of invariant random subgroups (IRSs), 
inspired by a similar criterion for 
permutation stability due to
Becker, Lubotzky and Thom.  
We apply our criterion to prove that derived subgroups of topological full groups 
of Cantor minimal subshifts are locally stable, 
using Zheng's classification of IRSs for these groups. 
This last result provides continuum-many 
groups which are locally stable, but not stable. 
\end{abstract}

\section{Introduction}

The existence of a non-sofic group is one of the major unsolved problems 
in group theory, and has been a key motivation behind the recent wave of interest in 
groups which are \emph{stable in permutations} (henceforth \emph{stable}). 
It is known that a sofic stable group must be residually finite, 
so to find a non-sofic group it suffices to find a group which is stable 
but not residually finite \cite{ArPa,GlRi}. 
The plausibility of this strategy was demonstrated by 
De Chiffre, Glebsky, Lubotzky and Thom \cite{ChGlLuTh}, 
who used the analogous notion of \emph{Frobenius stability} 
to produce examples of groups which are not \emph{Frobenius approximable}
(the latter property being analogous to soficity); 
see also \cite{ArPa2} for a construction in the 
context of ``constraint sofic approximations''. 

One difficulty with using stability in permutations as a path to finding a non-sofic group 
is that stability is a rather strong property for a group to satisfy, 
to the extent that the only groups known to be stable have also long been 
known to be sofic. 
In this paper we propose a weakening of stability which nonetheless 
retains the link with soficity: 
a sofic group satisfying our property, which we call \emph{local stability}, 
need not be residually finite, but must be locally embeddable into finite groups (LEF). 
Though it is harder to produce non-LEF groups (LEF being weaker than residual finiteness), 
this may be a price worth paying if locally stable groups prove to be abundant. 
To this end we exhibit many groups which are locally stable but not stable. 

\subsection{Terminology and statement of results} 
\label{TermSubsect}

Let $\Gamma$ be a finitely generated group. 
For $k \in \mathbb{N}$ let $d_k$ be the 
(normalized) Hamming metric on $\Sym(k)$, 
given by: 
\begin{equation*}
d_k (\sigma,\tau) = 1 -\frac{1}{k} 
\lvert \lbrace 1 \leq i \leq k 
: \sigma(i)=\tau(i) \rbrace \rvert
\end{equation*}
for $\sigma,\tau \in \Sym(k)$. 

\begin{defn}
Let $(\phi_n :\Gamma \rightarrow \Sym(k_n))$ be a sequence of functions. 
\begin{itemize}
\item[(i)] $(\phi_n)_{n}$ is an \emph{almost-homomorphism} 
if, for all $g, h \in \Gamma$, 
\begin{equation*}
d_{k_n} \big( \phi_n (g)\phi_n (h) , \phi_n (gh) \big) \rightarrow 0
\text{ as }n \rightarrow \infty;
\end{equation*}

\item[(ii)] $(\phi_n)_{n}$ is a \emph{partial homomorphism} 
if, for all $g, h \in \Gamma$, 
there exists $N > 0$ such that for all 
$n \geq N$, 
$\phi_n (gh) = \phi_n (g)\phi_n (h)$. 

\item[(iii)] $(\phi_n)_{n}$ is \emph{separating} 
if, for every $e \neq g \in \Gamma$, 
\begin{equation*}
d_{k_n} \big( \phi_n (g) , \id_{k_n} \big) \rightarrow 1
\text{ as }n \rightarrow \infty\text{.}
\end{equation*}
\end{itemize}
\end{defn}

\begin{defn} \label{SoficLEFDefn}
The group $\Gamma$ is \emph{sofic} if there is a 
separating almost-homomorphism 
$\phi_n :\Gamma \rightarrow \Sym(k_n)$ 
for some sequence of positive integers $(k_n)$. 
Similarly $\Gamma$ is 
\emph{LEF (locally embeddable into finite groups)} 
if it admits a separating partial homomorphism 
$\phi_n :\Gamma \rightarrow \Sym(k_n)$ 
for some $(k_n)$. 
\end{defn}

\begin{defn} \label{LocalStabDefn}
The group $\Gamma$ is \emph{locally stable} 
(respectively \emph{weakly locally stable}) if, 
for every almost-homomorphism 
(respectively every separating almost-homomorphism) 
$\phi_n :\Gamma \rightarrow \Sym(k_n)$, there is a partial 
homomorphism $\psi_n :\Gamma \rightarrow \Sym(n)$ such that 
for all $g \in \Gamma$, 
\begin{equation*}
d_{k_n} \big( \phi_n (g) , \psi_n (g) \big) \rightarrow 0
\text{ as }n \rightarrow \infty\text{.}
\end{equation*}
\end{defn}

It is immediate from these definitions that 
a sofic weakly locally stable group is LEF. 

\begin{rmrk}
If we further require in the above definition that 
$\psi_n :\Gamma \rightarrow \Sym(k_n)$ is a homomorphism for all $n$, 
then we arrive at the definition of a \emph{stable} 
(respectively \emph{weakly stable}) group. 
\end{rmrk}

\begin{rmrk}
We have: 
\begin{center}
$\begin{array}{ccc}
\text{Stable} & \Rightarrow & \text{Weakly Stable} \\
\Downarrow &  & \Downarrow \\
\text{Locally Stable} & \Rightarrow & \text{Weakly Locally Stable} \\
\end{array}$
\end{center}

\end{rmrk}

One of the consequences of our results will be that 
there are no other implications 
between these four properties. 
Our main result is the following. 

\begin{thm}[Theorem \ref{MainThm}] \label{IntroMainThm}
There is a continuum of pairwise nonisomorphic 
finitely generated groups, 
which are locally stable but not weakly stable. 
\end{thm}

To prove Theorem \ref{IntroMainThm}, 
we give a criterion for a finitely generated 
amenable group 
to be locally stable in terms of Invariant Random Subgroups (IRSs). 
Here our work is indebted to 
the influential paper of Becker, Lubotzky and Thom \cite{BeLuTh}, 
who gave a necessary and sufficient condition in terms of IRSs 
for an amenable group to be stable. 
Recall that an IRS of a discrete group $\Gamma$ is a probability measure on 
the Chabauty space $\Sub(\Gamma)$ of subgroups of $\Gamma$, 
which is invariant with respect to the conjugation action 
of $\Gamma$ on $\Sub(\Gamma)$. 
We denote by $\IRS(\Gamma)$ the space of all invariant random subgroups 
of $\Gamma$, equipped with the weak$^{\ast}$ topology. 
For $\Gamma$ a group generated by $d$ elements, 
we fix an epimorphism $\mathbb{F} \twoheadrightarrow \Gamma$ 
from a rank-$d$ free group $\mathbb{F}$ onto $\Gamma$. 
Then $\IRS (\Gamma)$ is naturally a subspace of $\IRS (\mathbb{F})$. 

\begin{thm}[Theorem \ref{IRSMainThm}] \label{IntroAmenMainThm}
Suppose $\Gamma$ is a finitely generated amenable group. 
Then $\Gamma$ is locally stable if and only if, 
for every IRS $\mu$ of $\Gamma$, 
there exists a sequence of $d$-marked finite groups $\Delta_n$ 
converging to $\Gamma$ in the space of marked groups, 
and an IRS $\nu_n$ of $\Delta_n$ such that 
$\nu_n \rightarrow \mu$ in $\IRS(\mathbb{F})$. 
\end{thm}

It is instructive to compare Theorem \ref{IntroAmenMainThm} with 
Theorem 1.3 of \cite{BeLuTh}. There it was proved that a finitely generated 
amenable group $\Gamma$ is stable if and only if every IRS of $\Gamma$ 
is the limit of atomic IRSs of $\Gamma$ supported on finite-index subgroups. 
Our criterion is weaker in the sense that, 
although the $\nu_n$ are atomic and supported on finite-index subgroups 
of $\mathbb{F}$ (we may view the $\nu_n$ as elements of $\IRS(\mathbb{F})$ 
since the $\Delta_n$ are also quotients of $\mathbb{F}$), 
they need not lie in $\IRS(\Gamma)$, since the $\Delta_n$ 
need not be quotients of $\Gamma$. 
If $\Gamma$ is additionally assumed to be finitely presented, however, 
then the $\Delta_n$ are eventually quotients of $\Gamma$, 
and we recover the Becker-Lubotzky-Thom criterion in this special case. 

Our main application of Theorem \ref{IntroAmenMainThm}
is to a family of groups arising from topological dynamics. 

\begin{thm}[Theorem \ref{FullGrpMainThm}] \label{IntroFullGrpMainThm}
Let $\Gamma$ be the topological full group of a Cantor minimal subshift. 
Then the derived subgroup $\Gamma'$ of $\Gamma$ is locally stable. 
\end{thm}

Theorem \ref{IntroMainThm} is then immediate, 
since the groups $\Gamma'$ apprearing in 
Theorem \ref{IntroFullGrpMainThm} are known to fail to be weakly stable, to be finitely generated, 
and to encompass 
a continuum of isomorphism-types of groups. 
The IRSs of such groups were classified by Zheng \cite{Zheng}. 
Theorem \ref{IntroFullGrpMainThm} is proven 
by showing that the IRSs arising in her classification satisfy 
the conditions of Theorem \ref{IntroAmenMainThm}, 
and using the breakthrough result 
of Juschenko-Monod that topological 
full groups are amenable \cite{JusMon}. 
To apply Theorem \ref{IntroAmenMainThm} 
to the IRSs described in Zheng's result, 
we need to produce suitable sequences of marked finite groups $\Delta_n$, and this will be achieved using the arguments of \cite{GrigMedy}, 
where it was proved that topological full groups are LEF, via the construction of 
partial actions on Kakutani-Rokhlin partitions. 

We have one other example of a finitely generated 
group which is locally stable but not weakly stable. 

\begin{thm}[Theorem \ref{AltEnrLocStabThm}] \label{IntroAltEnrThm}
Let $\mathcal{A}(\mathbb{Z})$ be the group of 
permutations of the set $\mathbb{Z}$ 
generated by all finitely supported even permutations 
and the image of the regular representation 
of the group $\mathbb{Z}$. 
Then $\mathcal{A}(\mathbb{Z})$ is a finitely generated 
group which is locally stable but not stable. 
\end{thm}

Although this latter result could also be viewed 
as an application of Theorem \ref{IntroAmenMainThm} 
(see Example \ref{AltEnrichIRSEx}), 
we shall in fact provide an alternative proof, 
using the fact that $\mathcal{A}(\mathbb{Z})$ 
may be obtained as the limit of a directed system 
of stable groups. 
The groups appearing in our directed system are drawn from a family constructed by B.H. Neumann \cite{Neum}; 
the fact of their stability is a result 
of Levit-Lubotzky \cite{LeLu}. 

In view of the relevance of local stability 
to the search for a non-sofic group, 
it is disappointing that all the new 
(that is, non-stable) examples of locally stable 
groups we produce are amenable (hence sofic), 
and indeed that amenability plays a key r\^{o}le 
in our arguments (via Theorem \ref{IntroAmenMainThm}). 
Nevertheless, it is intriguing to note that 
the topological full groups of $\mathbb{Z}$-actions 
studied in Theorem \ref{IntroFullGrpMainThm} 
have close relatives which are not known to be sofic. 
For instance Thompson's group $V$ is the 
topological full group of the action 
of a finitely generated group (namely of $V$ itself). 
We should note that, since $V$ is finitely presented, 
it is locally stable iff it is stable 
(see Lemma \ref{FPStabLem} below). 
Nevertheless, given the structural similarities between 
$V$ and the groups $\Gamma '$ 
arising in Theorem \ref{IntroFullGrpMainThm}, 
our results can be seen as evidence that 
$V$ is stable. 
If this were the case, then $V$ would be non-sofic. 

The paper is structured as follows. 
After fixing our notation and making some elementary observations about LEF groups, we characterize 
(weak) local stability, as defined above, 
in terms of local solutions to ``stability challenges''; 
stability properties of systems of equations, 
and lifting properties of homomorphisms 
to metric ultraproducts of finite symmetric groups. 
This is the subject of Subsection \ref{CharacSubsect}. 
In Subsection \ref{FirstPropSubsect} we show that 
(weak) stability and weak local stability 
are equivalent in the class of finitely presented groups, and that weak local stability and 
LEF are equivalent in the class of amenable groups. 
As a consequence, we exhibit an example 
of a weakly stable group which is not locally stable. 
In Section \ref{MarkedGrpsSect} we recall some 
basic material on the space of marked groups, 
to be used subsequently, 
and prove Theorem \ref{IntroAltEnrThm}. 
In Section \ref{IRSSect} we provide relevant background 
on invariant random subgroups, and prove Theorem 
\ref{IntroAmenMainThm}. 
In Section \ref{TFGSect} we discuss topological 
full groups, prove Theorem \ref{IntroFullGrpMainThm}, 
and deduce Theorem \ref{IntroMainThm}. 
We conclude with some open questions 
and suggested directions for future research. 

\section{Preliminaries}

\subsection{Notation}

In this paper, all group actions are on the left. 
For $\Gamma$ a group; $S \subseteq \Gamma$ 
a generating set and $n \in \mathbb{N}$, 
$B_S (n)$ denotes the closed ball of radius $n$ 
around the identity in the word-metric 
induced by $S$ on $\Gamma$. 
For $n$ a positive integer, 
let $[n]=\lbrace 1,\ldots,n\rbrace$ 
and let $\lBrack n \rBrack 
= \lbrace m \in \mathbb{Z} : \lvert m \rvert \leq n \rbrace$. 

\subsection{LEF groups}

The definition of LEF we have given in 
Definition \ref{SoficLEFDefn} above is non-standard. 
The class of LEF groups is more usually defined in 
terms of ``local embeddings'' into finite groups, 
as follows. 

\begin{defn}
Let $\Gamma$ and $\Delta$ be groups and let 
$A \subseteq \Gamma$. 
An injective function $\phi:A \rightarrow \Delta$ is a 
\emph{local embedding} if, for all $g,h \in A$, 
if $gh \in A$ then $\phi(gh)=\phi(g)\phi(h)$. 
\end{defn}

\begin{propn} \label{LEFStandardProp}
A countable group $\Gamma$ is LEF iff, 
for all finite $A \subseteq \Gamma$, 
there exists a finite group $Q$ and a 
local embedding $\phi : A \rightarrow Q$. 
\end{propn}

\begin{proof}
Let $(\psi_n : \Gamma \rightarrow \Sym(k_n))_n$ 
be a partial homomorphism and let $A \subseteq \Gamma$ 
be finite. Then for all $n$ sufficiently large, 
the restriction of $\psi_n$ to $A$ 
is a local embedding. 

Conversely suppose that $(A_n)$ is an ascending 
sequence of finite subsets of $\Gamma$ 
with union $\Gamma$, let $Q_n$ be a finite group 
affording a local embedding 
$\phi_n : A_n \rightarrow Q_n$, 
and let $\rho_n : Q_n\rightarrow\Sym (\lvert Q_n \rvert)$ 
be the regular representation of $Q_n$. 
Then any function 
$\psi_n :\Gamma\rightarrow \Sym(\lvert Q_n\rvert)$ 
agreeing with $\rho_n \circ \phi_n$ on $A_n$ 
is a separating partial homomorphism. 
\end{proof}

\subsection{Characterizations of local stability} 
\label{CharacSubsect}

Recall that local stability was defined in 
Definition \ref{LocalStabDefn}. 
Throughout this Subsection, $\Gamma$ is 
a group, generated by a finite set 
$S = \lbrace s_1 , \ldots , s_d \rbrace$; 
$\mathbb{F}$ is the free group on basis $S$, 
and $\pi : \mathbb{F} \rightarrow \Gamma$ 
is the standard epimorphism. 
Where there is no risk of confusion, 
we shall not distinguish between the free basis 
$S$ for $\mathbb{F}$ 
and its image under $\pi$ in $\Gamma$. 

\subsubsection{Stability challenges} \label{StabilityChallSubsub}

\begin{defn}
Let $X$ and $Y$ be finite $\mathbb{F}$-sets, with $\lvert X\rvert =\rvert Y\rvert$. 
For $f: X \rightarrow Y$ a bijection, set: 
\begin{equation*}
\lVert f \rVert_{\gen} = 
\frac{1}{\lvert S \rvert} \sum_{s \in S} \Prob_{x \in X} 
\Big[f\big(s(x)\big)\neq s\big(f(x)\big)\Big]
\end{equation*}
(where $x \in X$ is uniformly distributed). 
For $X$ and $Y$ finite $\mathbb{F}$-sets with $\lvert X \rvert = \lvert Y \rvert$, 
set $d_{\gen}(X,Y)$ to be the minimal value of $\lVert f \rVert_{\gen}$, 
as $f$ ranges over all bijections from $X$ to $Y$. 
\end{defn}

\begin{defn}
A \emph{stability challenge} for $\Gamma$ is a 
sequence $(X_n)$ of finite $\mathbb{F}$-sets 
such that, for all $r \in \ker (\pi)$, 
\begin{equation*}
\Prob_{x \in X_n} [r(x)\neq x] \rightarrow 0 \text{ as }n \rightarrow \infty. 
\end{equation*}
The stability challenge $(X_n)$ is \emph{separating} if, 
for all $w \in \mathbb{F} \setminus \ker(\pi)$, 
\begin{equation*}
\Prob_{x \in X_n} [w(x)\neq x] \rightarrow 1 \text{ as }n \rightarrow \infty. 
\end{equation*}
We call a stability challenge $(Y_n)$ for $\Gamma$ a 
\emph{local $\Gamma$-set} if every $r \in \ker (\pi)$ 
satisfies the stronger condition: 
\begin{center}
$\lbrace y \in Y_n: r(y)\neq y\rbrace = \emptyset$ 
for all sufficiently large $n$. 
\end{center}
A \emph{local solution} to the stability challenge $(X_n)$ 
is a local $\Gamma$-set $(Y_n)$ such that 
$\lvert X_n \rvert = \lvert Y_n \rvert$ for all $n$, 
and $d_{\gen}(X_n,Y_n) \rightarrow 0$ as $n \rightarrow \infty$. 
\end{defn}

For $k \in \mathbb{N}$, the Hamming metric $d_k$ 
on $\Sym(k)$ was defined in Subsection \ref{TermSubsect}. 
For an arbitrary finite set $X$, 
the Hamming metric $d_X$ on $\Sym(X)$ 
is defined in a precisely analogous way. 
In \cite{BeLuTh} Definitions 3.11 and 7.3, 
stable groups 
are characterized in terms of ``solutions'' 
to stability challenges. 
In the same spirit, we have the following. 

\begin{propn} \label{StabChallProp}
The group $\Gamma$ is locally stable iff 
every stability challenge for $\Gamma$ has a local solution. 
Similarly, $\Gamma$ is weakly locally stable iff 
every separating stability challenge for $\Gamma$ 
has a local solution. 
\end{propn}

\begin{proof}
Suppose that $\Gamma$ is locally stable, 
and let $(X_n)$ be a stability challenge for $\Gamma$, 
inducing homomorphisms $\tilde{\phi}_n : \mathbb{F} \rightarrow \Sym(X_n)$. 
Then for every $r \in \ker(\pi)$, 
\linebreak$d_{X_n} (\tilde{\phi}_n(r),\id_{X_n}) \rightarrow 0$ as $n \rightarrow \infty$. 
For each $g \in \Gamma$, pick some $w(g) \in \pi ^{-1} (g)$, 
with $w (e) = e$ and $w(s_i)=s_i$. 
Define $\phi_n : \Gamma \rightarrow \Sym(X_n)$ by 
$\phi_n (g) = \tilde{\phi}_n (w(g))$. Then for $g,h \in \Gamma$, 
$w(g)w(h)w(gh)^{-1} \in \ker(\pi)$, so: 
\begin{align*}
d_{X_n} \big( \phi_n (g) \phi_n (h),\phi_n (gh) \big)
& = d_{X_n} \big( \phi_n (g) \phi_n (h)\phi_n (gh)^{-1},\id_{X_n} \big) \\
& = d_{X_n} \big( \tilde{\phi}_n \big( w(g)w(h)w(gh)^{-1} \big),\id_{X_n} \big) \\
& \rightarrow 0 \text{ as }n \rightarrow \infty
\end{align*}
Thus $(\phi_n)$ is an almost-homomorphism; by local stability 
there is a partial homomorphism $(\psi_n)$ 
satisfying the conclusion of Definition \ref{LocalStabDefn}. 
In particular $\psi_n (e) = \id_{X_n}$ for all $n$ sufficiently large. 
Define an action of $\mathbb{F}$ on $Y_n = X_n$ by: 
\begin{center}
$w \cdot y = (\psi_n \circ \pi)(w)[y]$ 
for all $w \in \mathbb{F}$ and $y \in Y_n$. 
\end{center}
Then each $r \in \ker (\pi)$ acts trivially on $Y_n$ 
for all sufficiently large $n$. 
Thus let $f_n = \id : X_n \rightarrow Y_n$, 
so that if $x \in X_n$ and $s \in S$ are such that 
$f_n (s \cdot x) \neq s \cdot f_n (x)$, 
then $\phi(s)[x]\neq \psi(s)[x]$. 
Thus $d_{\gen}(X_n,Y_n) \leq \lVert f_n \rVert_{\gen} \rightarrow 0$ 
as $n \rightarrow 0$. 

Conversely suppose that every stability challenge for $\Gamma$ has a local solution, 
and let $\phi_n : \Gamma \rightarrow \Sym(X_n)$ be an almost-homomorphism, 
so that $X_n$ is an $\mathbb{F}$-set via 
$w \cdot x = (\phi_n \circ \pi)(w)[x]$. 
Then $(X_n)$ is easily seen to be a stability challenge for $\Gamma$. 
Let $(Y_n)$ be a local $\Gamma$-set such that 
there exists a sequence of bijections 
$(f_n : X_n \rightarrow Y_n)$ with $\lVert f_n \rVert_{\gen} \rightarrow 0$ 
as $n \rightarrow \infty$. 
Define a function 
$\psi_n : \Gamma \rightarrow \Sym(X_n)$ by: 
\begin{equation*}
\psi_n (g)(x) = f_n ^{-1} (w(g) \cdot f_n (x))
\end{equation*}
(it is easily seen that each $\psi_n (g)$ is bijective). 
Then for $g,h \in \Gamma$, 
$w(g)w(h)w(gh)^{-1} \in \ker (\pi)$, 
and since $(Y_n)$ a local solution, 
for all $n$ sufficiently large, 
$w(g)w(h)w(gh)^{-1} \cdot y = y$ for all $y \in Y_n$, 
from which it easily follows that 
$\psi_n(g) \psi_n(h) = \psi_n(gh)$. 
Hence $(\psi_n)_n$ is a partial homomorphism. 
Moreover, the fact that 
$\lVert f_n \rVert_{\gen} \rightarrow 0$ 
implies that for every fixed $g \in \Gamma$, 
$d_{k_n} (\phi_n(g),\psi_n(g))\rightarrow 0$. 

The proof of the criterion for weak local stability 
is essentially identical: it suffices to note 
(in the forward direction) 
that if we start with a separating stability challenge 
$(X_n)$, then the almost-homomorphism $(\phi_n)$ we constructed is separating also, 
and (in the reverse direction) that 
if we start with a separating almost-homomorphism 
$(\phi_n)$, 
then the stability challenge $(X_n)$ 
we constructed is separating. 
\end{proof}

\subsubsection{Locally stable systems of equations}

For $\overline{\sigma} \in \Sym(k)^S$, 
let $\pi_{\overline{\sigma}} : \mathbb{F} \rightarrow \Sym(k)$ be the (unique) homomorphism extending 
$\overline{\sigma}$, and for $w \in \mathbb{F}$, 
write $w(\overline{\sigma})=\pi_{\overline{\sigma}}(w)$ 
(the evaluation of $w$ at $\overline{\sigma}$). 
For $R \subseteq \mathbb{F}$, we shall say that 
$\overline{\sigma}$ is a \emph{solution} to 
the system of equations $\lbrace r=e\rbrace_{r \in R}$ 
if $r(\overline{\sigma}) = \id_k$ for all $r \in R$. 
For $\epsilon > 0$, $\overline{\sigma}$ is an 
\emph{$\epsilon$-almost-solution} to 
$\lbrace r=e\rbrace_{r \in R}$ if, 
for all $r \in R$, 
$d_k (r(\overline{\sigma}),\id_k) < \epsilon$ 
(for short, we may also refer to an ``$\epsilon$-almost solution for $R$'' or a ``solution for $R$'', 
with the same meaning). 
We shall say that 
$\overline{\sigma},\overline{\tau} \in \Sym(k)^S$ 
are $\epsilon$-close if: 
\begin{equation*}
\sum_{s \in S} d_k \big( \overline{\sigma}(s),\overline{\tau}(s) \big) < \epsilon
\end{equation*}

\begin{defn}
For $R \subseteq \mathbb{F}$, 
we say that the system of equations 
$\lbrace r = e \rbrace_{r \in R}$ 
is \emph{locally stable in permutations} if, 
for all $\epsilon > 0$ and $R_0 \subseteq R$ finite, 
there exists $\delta > 0$ and $R_1 \subseteq R$ finite 
such that, for every $k \in \mathbb{N}$ and every 
$\overline{\sigma} \in \Sym(k) ^S$, 
if $\overline{\sigma}$ is a $\delta$-almost-solution 
to $\lbrace r = e \rbrace_{r \in R_1}$, 
then there is a solution 
$\overline{\tau} \in \Sym(k) ^S$ to 
$\lbrace r = e \rbrace_{r \in R_0}$ 
such that $\overline{\sigma}$ and $\overline{\tau}$ 
are $\epsilon$-close in $\Sym(k)$. 
\end{defn}

This definition is a generalization 
of the concept of a ``stable system of equations'' 
first introduced in \cite{GlRi}. 
In Lemma 3.12 of \cite{BeLuTh}, 
it is shown that the group $\Gamma$ is stable 
iff the system of equations 
$\lbrace r = e : r \in \ker (\pi) \rbrace$  is stable. 
The analogous statement for local stability 
is as follows. 

\begin{propn} \label{LocStabEqnProp}
The group $\Gamma$ is locally stable 
iff the system of equations 
$\lbrace r = e : r \in \ker (\pi) \rbrace$ 
is locally stable in permutations. 
\end{propn}

\begin{proof}
We use the characterization of local stability 
for groups in terms of local solutions to 
stability challenges from 
\S \ref{StabilityChallSubsub} above. 
Suppose first that the system 
$\lbrace r = e \rbrace_{r \in \ker (\pi)}$ is not 
locally stable in permutations. 
Then there exists $\epsilon > 0$ and 
$R \subseteq \ker(\pi)$ finite such that 
for all $n \in \mathbb{N}$, 
there exists $k_n \in \mathbb{N}$ and 
a $(1/n)$-almost-solution 
$\overline{\sigma}^{(n)} \in \Sym(k_n)^S$ 
to $\ker (\pi) \cap B_S (n)$, which is not 
$\epsilon$-close to any solution to $\lbrace r = e \rbrace_{r \in R}$. 
Let $X_n = [k_n]$ be an $\mathbb{F}$-set, 
via $w \cdot j = w (\overline{\sigma}^{(n)}) (j)$. 
Then $(X_n)$ is a stability challenge for $\Gamma$. 
If $(Y_n)$ is a local solution, 
and $f_n :X_n \rightarrow Y_n$ are bijections 
satisfying $\lVert f_n \rVert_n \rightarrow 0$ 
as $n \rightarrow \infty$, then, 
defining $\overline{\tau}^{(n)} \in \Sym(k_n)^S$ 
by $\overline{\tau}^{(n)}(s_i)(j)=f_n ^{-1}(s_i \cdot f_n (j))$, we have that for all $n$ sufficiently 
large, $\overline{\tau}^{(n)}$ is a solution 
for $\lbrace r = e \rbrace_{r \in R}$ 
which is $\epsilon$-close to $\overline{\sigma}^{(n)}$, 
contradiction. 

Conversely suppose that 
$\lbrace r = e \rbrace_{r \in \ker (\pi)}$ is 
locally stable in permutations. 
For each $m \in \mathbb{N}$, let $\delta_m > 0$ 
and $a_m \in \mathbb{N}$ be such that 
$\delta_m \rightarrow 0$; 
$(a_m)$ is an increasing sequence and every 
$\delta_m$-almost-solution to 
$\ker(\pi) \cap B_S (a_m)$ is $(1/m)$-close 
to a solution for $\ker(\pi) \cap B_S (m)$. 
Let $(X_n)$ be a stability challenge for $\Gamma$. 
Let $k_n = \lvert X_n \rvert$ and fix a bijection 
$g_n : [k_n] \rightarrow X_n$. 
Define $\overline{\sigma}^{(n)} \in \Sym(k_n)^S$ 
by $\overline{\sigma}^{(n)}(s_i)(j) 
= g_n ^{-1} (s_i \cdot g_n (j))$. 
Then there exists a strictly increasing sequence $(N_m)$ 
such that for all $n \geq N_m$, 
$\overline{\sigma}^{(n)}$ is a $\delta_m$-solution 
to $\ker(\pi) \cap B_S (a_m)$. 
For $N_m \leq n < N_{m+1}$, 
let $\overline{\tau}^{(n)}$ be a solution to 
$\ker(\pi) \cap B_S (m)$ which is $(1/m)$-close 
to $\overline{\sigma}^{(n)}$. 
Define $Y_n = [k_n]$, which is an $\mathbb{F}$-set 
via $w \cdot y = w (\overline{\tau}^{(n)})(y)$. 
Then $(Y_n)$ is a local $\Gamma$-set, 
and the bijections 
$f_n = g_n ^{-1} : X_n \rightarrow Y_n$ 
witness that $(Y_n)$ is a local solution to 
the stability challenge $(X_n)$. 
\end{proof}

\begin{rmrk}
\normalfont
It follows from Proposition \ref{LocStabEqnProp} 
that for $R \subseteq \mathbb{F}$, 
whether or not the system $\lbrace r=e : r\in R \rbrace$ 
is locally stable in permutations 
depends only on the isomorphism-type of 
the group presented by $\langle S | R \rangle$. 
The corresponding statement for stability 
follows similarly from Lemma 3.12 of \cite{BeLuTh}. 
The special case of finitely presented groups 
appeared already in \cite{ArPa}: as is noted in that paper, 
stability of a system of relations 
is unaffected by applying Tietze moves to the system. 
\end{rmrk}

We can likewise capture weak local stability in terms 
of solutions to equations. 
For the sake of simplicity, 
we restrict our attention to stability of families of 
equations constituting a normal subgroup of 
$\mathbb{F}$. 

\begin{defn}
For $R \subseteq\mathbb{F} \setminus\lbrace e\rbrace$, 
and $\delta > 0$, 
$\overline{\sigma} \in \Sym(k)^S$ 
is \emph{$(1-\delta)-$separating} for $R$ if, 
for all $r \in R$, 
\begin{center}
$d_k \big( r(\overline{\sigma}),\id_k\big) > 1-\delta$. 
\end{center}
For $N \vartriangleleft \mathbb{F}$, 
the system $\lbrace r = e : r \in N \rbrace$ is 
\emph{weakly locally stable in permutations} if, 
for all $\epsilon > 0$ and $R_0 \subseteq N$ finite, 
there exist $\delta > 0$ and finite subsets 
$R_1 \subseteq N$ and 
$R_1 ' \subseteq \mathbb{F} \setminus N$ 
such that every $\delta$-almost-solution 
to $\lbrace r=e : r \in R_1 \rbrace$, 
which is $(1-\delta)$-separating for $R_1 '$, 
is $\epsilon$-close to a solution 
for $\lbrace r=e : r \in R_0 \rbrace$. 
\end{defn}

\begin{propn} \label{WeakLocStabEqnProp}
The group $\Gamma$ is weakly locally stable 
iff the system of equations 
$\lbrace r = e : r \in \ker (\pi) \rbrace$ 
is locally stable in permutations. 
\end{propn}

\begin{proof}
This is much the same as the proof of Proposition 
\ref{LocStabEqnProp}. 
For the ``only if'' direction, 
we may assume that the tuple 
$\overline{\sigma}^{(n)} \in \Sym(k_n)^S$ 
described in the proof of 
Proposition \ref{LocStabEqnProp} 
is $(1-1/n)$-separating for $B_S(n)\setminus\ker(\pi)$. 
For the ``if'' direction, 
we choose $\delta_m$ and $a_m$ such that 
every $\delta_m$-almost-solution to 
$\ker(\pi) \cap B_S (a_m)$, 
which is $(1-\delta_m)$-separating for 
$B_S (a_m) \setminus \ker(\pi)$, 
is $(1/m)$-close 
to a solution for $\ker(\pi) \cap B_S (m)$, 
and argue as in the proof of Proposition 
\ref{LocStabEqnProp}. 
\end{proof}

\subsubsection{Liftings properties of homomorphisms to ultraproducts}

Let $\mathbf{k} = (k_n)_n$ be an increasing sequence of positive integers. Let: 
\begin{equation*}
G_{\mathbf{k}} = \prod_{n} \Sym(k_n)
\end{equation*}
be the Cartesian product of the finite groups $\Sym(k_n)$. 
Let $\mathcal{U}$ be a non-principal ultrafilter on $\mathbb{N}$. 
Define the normal subgroups: 
\begin{align*}
N_{\mathcal{A},\mathbf{k}} &
= \big\lbrace (\sigma_n)\in G_{\mathbf{k}}:
\lbrace n\in \mathbb{N}:\sigma_n=\id_{k_n}\rbrace\in\mathcal{U} \big\rbrace \\
N_{\mathcal{M},\mathbf{k}} &
= \big\lbrace (\sigma_n)\in G_{\mathbf{k}}:
\forall\epsilon>0, \lbrace n \in \mathbb{N} : d_{k_n}(\sigma_n,\id_{k_n})<\epsilon \rbrace \in\mathcal{U} \big\rbrace
\end{align*}
of $G_{\mathbf{k}}$ 
(so that $N_{\mathcal{A},\mathbf{k}} \leq N_{\mathcal{M},\mathbf{k}}$) 
and define $G_{\mathcal{A},\mathbf{k}} = G_{\mathbf{k}}/N_{\mathcal{A},\mathbf{k}}$, 
$G_{\mathcal{M},\mathbf{k}} = G/N_{\mathcal{M},\mathbf{k}}$ 
(called, respectively, the \emph{algebraic ultraproduct} 
and \emph{metric ultraproduct} of the $\Sym(k_n)$). 
Let $\pi_{\mathcal{A}} : G_{\mathbf{k}} \twoheadrightarrow G_{\mathcal{A},\mathbf{k}}$, 
$\pi_{\mathcal{M}} : G_{\mathbf{k}} \twoheadrightarrow G_{\mathcal{M},\mathbf{k}}$ 
and $\pi_{\mathcal{M},\mathcal{A}} : G_{\mathcal{A},\mathbf{k}} \twoheadrightarrow G_{\mathcal{M},\mathbf{k}}$ 
be the natural epimorphisms. 
For any $(\sigma_n),(\tau_n) \in G_{\mathbf{k}}$, 
there is a well-defined limit 
$\tilde{d}_{\mathcal{M},\mathbf{k}} \big( (\sigma_n),(\tau_n) \big) = 
\lim_{n \rightarrow \mathcal{U}} d_{k_n} (\sigma_n,\tau_n) \in [0,1]$ of $d_{k_n} (\sigma_n,\tau_n)$ 
along $\mathcal{U}$. Then $\tilde{d}_{\mathcal{M},\mathbf{k}}$ 
is a pseudometric on $G_{\mathbf{k}}$, 
which descends to a well-defined metric 
$d_{\mathcal{M},\mathbf{k}}$ on 
$G_{\mathcal{M},\mathbf{k}}$. 

\begin{defn}
For $\Gamma$ a group, a \emph{sofic representation} 
of $\Gamma$ is a homomorphism 
$\Phi :\Gamma \rightarrow G_{\mathcal{M},\mathbf{k}}$ 
such that, for all $e \neq g \in \Gamma$, 
$d_{\mathcal{M},\mathbf{k}}\big(\Phi(g),e\big)=1$. 
\end{defn}

A countable group $\Gamma$ is sofic iff it admits 
a sofic representation (see \cite{EleSzaHyp} 
Section 2 for a proof). 
There is also a characterization of LEF in terms of 
embeddings into algebraic ultraproducts; 
see \cite{CooChSi} Section 7.2 
and the references therein. 
We include a proof of the exact formulation we need 
for the reader's convenience. 

\begin{lem} \label{LEFUltraprodLem}
A countable group $\Gamma$ is LEF iff 
there exists an increasing sequence 
$\mathbf{k} = (k_n)_n$ and a homomorphism 
$\Phi : \Gamma \rightarrow G_{\mathcal{A},\mathbf{k}}$ 
such that $\pi_{\mathcal{M},\mathcal{A}} \circ \Phi$ 
is a sofic representation. 
\end{lem}

\begin{proof}
Given a separating partial homomorphism 
$(\phi_n : \Gamma \rightarrow \Sym(k_n))$, 
let $\mathbf{k} = (k_n)_n$ and define the function 
$\tilde{\Phi} : \Gamma \rightarrow G_{\mathbf{k}}$ 
by $\tilde{\Phi}(g) = (\phi_n (g))_n$. 
Then $\Phi = \pi_{\mathcal{A}} \circ \tilde{\Phi}$ 
has the desired properties. 

Conversely suppose such a homomorphism $\Phi$ exists, 
let $A \subseteq \Gamma$ be a finite set, 
and for each $g \in A$, let 
$\tilde{\Phi}(g) \in G_{\mathbf{k}}$ be a lift of 
$\Phi (g)$, so that $\tilde{\Phi}(g)_n \in \Sym(k_n)$. 
For each $g,h \in A$ satisfying $gh \in A$, 
we have 
$\tilde{\Phi}(g)\tilde{\Phi}(h)\tilde{\Phi}(gh)^{-1} 
\in N_{\mathcal{A},\mathbf{k}}$. 
Moreover for each $e \neq g \in A$, 
since $d_{\mathcal{M},\mathbf{k}} \big( (\pi_{\mathcal{M},\mathcal{A}} \circ \Phi)(g),e \big)=1$, 
\begin{center}
$\lbrace n \in \mathbb{N} : 
d_{k_n} ( \tilde{\Phi}(g)_n,\id_{k_n} ) \geq 1/2 \rbrace \in \mathcal{U}$. 
\end{center}
In particular, there exists $n \in \mathbb{N}$ 
such that $g \mapsto \tilde{\Phi}(g)_n$ 
defines a local embedding. 
We apply Proposition \ref{LEFStandardProp} 
to conclude. 
\end{proof}

In \cite{ArPa} Sections 4 and 6, 
it is proved that $\Gamma$ is (weakly) stable iff 
for every $\mathbf{k}$, 
every homomorphism 
(respectively every sofic representation) 
from $\Gamma$ to $G_{\mathcal{M},\mathbf{k}}$ 
may be lifted to $G_{\mathbf{k}}$. 
Similarly, we have the following. 

\begin{propn} \label{UltraProdLiftProp}
The group $\Gamma$ is (weakly) locally stable iff, 
for every $\mathbf{k}$ 
and every homomorphism 
(respectively every sofic representation) 
$\Phi:\Gamma \rightarrow G_{\mathcal{M},\mathbf{k}}$, 
there is a homomorphism $\hat{\Phi} : \Gamma \rightarrow G_{\mathcal{A},\mathbf{k}}$ 
such that 
$\pi_{\mathcal{M},\mathcal{A}} \circ \hat{\Phi} = \Phi$. 
\end{propn}

\begin{proof}
We shall use Proposition \ref{LocStabEqnProp}. 
Suppose first that 
$\lbrace r=e \rbrace_{r \in \ker(\pi)}$ 
is locally stable in permutations. 
Let $\Phi:\Gamma \rightarrow G_{\mathcal{M},\mathbf{k}}$ 
be a homomorphism; let 
$\tilde{\Phi}:\mathbb{F}\rightarrow G_{\mathbf{k}}$ 
be a lift of $(\Phi \circ \pi):\mathbb{F}
\rightarrow G_{\mathcal{M},\mathbf{k}}$ 
to $G_{\mathbf{k}}$ 
(so that $\tilde{\Phi}(\ker(\pi)) 
\leq N_{\mathcal{M},\mathbf{k}}$), 
and let $\tilde{\Phi}_n = p_n \circ \tilde{\Phi}$, 
where $p_n : G_{\mathbf{k}} \rightarrow \Sym(k_n)$ 
is projection to the $n$th co-ordinate. 
For any finite $R \subseteq \ker(\pi)$ 
and any $\delta > 0$, 
\begin{center}
$\lbrace n \in \mathbb{N} 
: d_{k_n} \big( \tilde{\Phi}_n(r),\id_{k_n} \big) 
< \delta \text{ for all } r \in R \rbrace \in \mathcal{U}$. 
\end{center}
For each $m \in \mathbb{N}$, let $\delta_m > 0$ 
and $a_m \in \mathbb{N}$ be such that 
$\delta_m \rightarrow 0$; 
$(a_m)$ is an increasing sequence and every 
$\delta_m$-almost-solution to 
$\ker(\pi) \cap B_S (a_m)$ is $(1/m)$-close 
to a solution for $\ker(\pi) \cap B_S (m)$. 
Let: 
\begin{center}
$I_m = \big\lbrace n \in \mathbb{N} 
: d_{k_n} \big( \tilde{\Phi}_n(r),\id_{k_n} \big) 
< \delta_m \text{ for all }r \in \ker(\pi) \cap B_S (a_m) \big\rbrace \in \mathcal{U}$. 
\end{center}
Then $I_{m+1} \subseteq I_m$ and: 
\begin{equation*}
\bigcap_{m \in \mathbb{N}} I_m = \lbrace n \in \mathbb{N} 
: \tilde{\Phi}_n (r)=e \text{ for all }r \in \ker(\pi) \rbrace. 
\end{equation*}
For $n \in I_m \setminus I_{m+1}$, 
there is a solution $\tilde{\Psi}_n \in \Sym (k_n) ^S$ 
to $\ker(\pi) \cap B_S (m)$ 
which is $(1/m)$-close to $\tilde{\Phi}_n$. 
Extend $\tilde{\Psi}_n$ to $\mathbb{F}$ 
and define 
$\tilde{\Psi}:\mathbb{F}\rightarrow G_{\mathbf{k}}$ 
by $p_n \circ \tilde{\Psi} = \tilde{\Psi}_n$. Then: 
every $r \in \ker (\pi)$ lies in $B_S (m)$ 
for all $m$ sufficiently large, 
so $\lbrace n \in \mathbb{N} : \tilde{\Psi}_n(r)=\id_{k_n} \rbrace \in \mathcal{U}$, 
hence $\tilde{\Psi} (\ker(\pi)) \leq N_{\mathcal{A},\mathbf{k}}$, 
and $\tilde{\Psi}$ descends to 
$\Psi : \Gamma \rightarrow G_{\mathcal{A},\mathbf{k}}$. 
Finally, for all $\epsilon > 0$, 
\begin{equation*}
\lbrace n \in \mathbb{N} 
: d_{k_n} \big( \tilde{\Phi}_n(s),\tilde{\Psi}_n(s) \big) < \epsilon \text{ for all }s \in S \rbrace 
\in \mathcal{U},
\end{equation*}
so $\pi_{\mathcal{M},\mathcal{A}} \circ \Psi$ 
agrees with $\Phi$ on $S$, 
hence on $\Gamma$, and $\hat{\Phi}=\Psi$ is as desired. 

Conversely, suppose that 
$\lbrace r=e \rbrace_{r \in \ker(\pi)}$ 
is not locally stable in permutations. 
Then there exists $\epsilon > 0$ and 
$R \subseteq \ker(\pi)$ finite such that 
for all $n \in \mathbb{N}$, 
there exists $k_n \in \mathbb{N}$ and 
a $(1/n)$-almost-solution 
$\overline{\sigma}^{(n)} \in \Sym(k_n)^S$ 
to $\ker (\pi) \cap B_S (n)$, which is not 
$\epsilon$-close to any solution to $R$. Define 
$\tilde{\Phi}:\mathbb{F} \rightarrow G_{\mathbf{k}}$ 
by $\tilde{\Phi} (s) = (\overline{\sigma}^{(n)}(s))_n$. 
Then for every $r \in \ker(\pi)$, 
and all $n$ sufficiently large, 
$d_{k_n} (r (\overline{\sigma}^{(n)}),\id_{k_n}) < 1/n$, 
so $\tilde{\Phi} (r) \in N_{\mathcal{M},\mathbf{k}}$, 
and $\tilde{\Phi}$ descends to 
$\Phi : \Gamma \rightarrow G_{\mathcal{M},\mathbf{k}}$. 
If $\Phi$ lifts to 
$\hat{\Phi}:\Gamma
\rightarrow G_{\mathcal{A},\mathbf{k}}$, 
and $\overline{\tau} ^{(n)} \in \Sym (k_n)^S$ is such that 
$\hat{\Phi} (\pi(s)) = (\overline{\tau} ^{(n)}(s))_n N_{\mathcal{A},\mathbf{k}}$ for all $s \in S$, 
then $\big( \overline{\sigma}^{(n)}(s)\overline{\tau} ^{(n)}(s)^{-1} \big)_n \in N_{\mathcal{M},\mathbf{k}}$ 
for $s \in S$. 
Hence there exists $n$ for which 
$\overline{\sigma}^{(n)}$ and $\overline{\tau} ^{(n)}$ are $\epsilon$-close and 
$r (\overline{\tau} ^{(n)}) = e$ for all $r \in R$, 
contradiction. 

The argument for weak local stability is much the same, 
using Proposition \ref{WeakLocStabEqnProp} 
in the forward direction. 
\end{proof}

\subsection{First properties of the class of locally stable groups} \label{FirstPropSubsect}

All stable groups are locally stable. 
This includes all finite groups \cite{GlRi}; 
all polycyclic-by-finite groups 
and the Baumslag-Solitar groups $\BS(1,n)$  \cite{BeLuTh}; 
Grigorchuk's group and the Gupta-Sidki $p$-groups 
\cite{Zheng}, and the (restricted, regular) 
wreath product of any two finitely generated 
abelian groups \cite{LevLubWP}. 
Note that all groups listed above are amenable; 
finite-rank nonabelian free groups provide a class of 
nonamenable groups which are easily seen to be stable: 
a free group admits a finite presentation with 
no relations, and the empty set of equations 
is trivially stable. 

We continue to let $S$ be a finite generating set for 
the group $\Gamma$, $\mathbb{F}$ be the free group on basis $S$, 
and $\pi : \mathbb{F} \rightarrow \Gamma$ 
be the associated epimorphism. 

\begin{lem} \label{FPStabLem}
Suppose $\Gamma$ is finitely presented and (weakly) locally stable. 
Then $\Gamma$ is (weakly) stable. 
\end{lem}

\begin{proof}
It suffices to show that, 
for every partial homomorphism 
$\phi_n :\Gamma \rightarrow\Sym(k_n)$, 
there is a sequence of homomorphisms 
$\psi_n :\Gamma \rightarrow\Sym(k_n)$ such that, 
for all $g \in \Gamma$, 
$d_{k_n} \big( \phi_n(g),\psi_n(g) \big) \rightarrow 0$ 
as $n \rightarrow \infty$. 
Indeed, we shall produce $\psi_n$ such that 
$\phi_n(g) = \psi_n(g)$ 
for all sufficiently large $n$. 

Let $R \subseteq \mathbb{F}$ be a finite set which normally generates $\ker (\pi)$. 
Given $m > 0$, let $N = N(m) > 0$ be sufficiently large that for all $g,h \in \Gamma$ 
of word-length at most $m$ with respect to $S$, 
and for all $n \geq N$, $\phi_n(gh) = \phi_n(g)\phi_n(h)$. 
In particular, for every $w \in \mathbb{F}$ 
of length at most $m$, 
we have: 
\begin{equation} \label{FPEqn}
\phi_n (\pi(w)) = w (\overline{\sigma}^{(n)}), 
\end{equation}
where $\overline{\sigma}^{(n)} = (\phi_n(s))_{s\in S} \in \Sym(k_n)^S$. 
In particular, for all $r \in R$, 
$r(\overline{\sigma}^{(n)})=\id_{k_n}$, 
and we have a well-defined homomorphism 
$\psi_n : \Gamma \rightarrow \Sym(k_n)$ 
sending $s$ to $\phi_n(s)$. 
By (\ref{FPEqn}) $\phi_n$ and $\psi_n$ agree 
on $B_S (m) \subseteq \Gamma$, as desired. 
\end{proof}

We recall the key observation made following 
Definition \ref{LocalStabDefn} above, 
which follows immediately from the relevant definitions. 
The analogous statement for stability 
appeared \cite{GlRi} Theorem 2 
and \cite{ArPa} Theorem 7.2. 

\begin{lem} \label{soficwlslem}
Suppose $\Gamma$ is sofic and weakly locally stable. 
Then $\Gamma$ is LEF. 
\end{lem}

Within the class of amenable groups, 
there is a simple necessary-and-sufficient 
condition for weak stability. 

\begin{thm}[\cite{ArPa} Theorem 7.2 (iii)] \label{ArzPauThm}
Suppose $\Gamma$ is amenable. 
Then $\Gamma$ is weakly stable iff it is 
residually finite. 
\end{thm}

We prove a similar criterion for weak local stability 
among amenable groups. 

\begin{lem} \label{amenwlslem}
Suppose $\Gamma$ is amenable. 
Then $\Gamma$ is weakly locally stable iff it is LEF. 
\end{lem}

\begin{proof}
Every amenable group is sofic, as a result 
the ``only if'' direction follows from 
Lemma \ref{soficwlslem}. 
For the ``if'' direction, 
we shall use the criterion in terms of 
ultraproducts from Proposition \ref{UltraProdLiftProp}. 
Since $\Gamma$ is LEF, by Lemma \ref{LEFUltraprodLem}, 
there exists an increasing sequence $\mathbf{k}=(k_n)$ 
and a homomorphism 
$\psi : \Gamma \rightarrow G_{\mathcal{A},\mathbf{k}}$ 
such that $\pi_{\mathcal{M},\mathcal{A}} \circ \psi$ 
is a separating homomorphism.  
Let $\phi : \Gamma \rightarrow G_{\mathcal{M},\mathbf{k}}$ 
be any separating homomorphism. 
By \cite{ElSz}, 
since $\Gamma$ is amenable, 
if $\phi_1 ,\phi_2 :\Gamma \rightarrow G_{\mathcal{M},\mathbf{k}}$ 
are separating homomorphisms, 
then there exists $h \in G_{\mathcal{M},\mathbf{k}}$ 
such that for all $g \in \Gamma$, 
$\phi_2 (g) = h \phi_1 (g) h^{-1}$. 
Apply this result to $\phi_1 = \phi$ and 
$\phi_2 = \pi_{\mathcal{M},\mathcal{U}} \circ \psi$ 
to obtain a corresponding $h \in G_{\mathcal{M},\mathbf{k}}$. 
Letting $\tilde{h} \in \pi_{\mathcal{M},\mathcal{A}} ^{-1} (h)$, 
$\tilde{\phi} : g \mapsto \tilde{h}^{-1} \psi(g) \tilde{h}$ is a homomorphism 
satisfying $\pi_{\mathcal{M},\mathcal{A}} \circ \tilde{\phi} = \phi$. 
By Proposition \ref{UltraProdLiftProp}, 
$\Gamma$ is weakly stable. 
\end{proof}

\begin{coroll}
There exist groups with the following properties: 
\begin{itemize}
\item[(i)] A finitely generated weakly locally stable group which is 
not weakly stable. 
\item[(ii)] A finitely generated group which 
is not weakly locally stable. 
\end{itemize}
\end{coroll}

\begin{proof}
\begin{itemize}
\item[(i)] By Lemma \ref{amenwlslem} and 
Theorem \ref{ArzPauThm}, 
any finitely generated amenable, 
LEF group which is not residually finite will do. 
The wreath product 
$\Alt(5) \wr \mathbb{Z}$ is such a group: 
it is amenable, being (locally finite)-by-abelian, 
and is LEF by Theorem 2.4 (ii) of \cite{VerGor}. 

\item[(ii)] By Lemma \ref{soficwlslem} 
any finitely generated sofic group which is not LEF will do. 
The Baumslag-Solitar group $\BS(2,3)$ is such a group: 
it is not LEF by Corollary 4 in Section 2.2 of \cite{VerGor} 
(and the comment following); 
its soficity is explained, for instance in 
Example 4.6 of \cite{Pestov}. 
\end{itemize}
\end{proof}

\begin{thm} \label{AbelsGroupThm}
There is a finitely presented soluble group which is weakly stable 
but not locally stable. 
\end{thm}

\begin{proof}
Let $p$ be a prime and let 
$A_p \leq \GL_4 (\mathbb{Q})$ be 
the $p$-Abels group; see \cite{BeLuTh}[Corollary 8.7]. 
Then $A_p$ is: 
\begin{itemize}
\item[(i)] finitely presented; 

\item[(ii)] linear over $\mathbb{Q}$ (and being 
finitely generated, is therefore residually finite, 
by Mal'cev's Theorem); 

\item[(iii)] soluble (hence amenable); 

\item[(iv)] not stable. 
\end{itemize}
By (i) and (iv) we may apply Lemma 
\ref{FPStabLem} to conclude that 
$A_p$ is not locally stable. 
By (ii) and (iii) we may apply Theorem \ref{ArzPauThm} 
and deduce that $\Gamma$ is weakly stable. 
\end{proof}

\section{Limits in the space of marked groups} 
\label{MarkedGrpsSect}

The space $\mathcal{G}_d$ of marked $d$-generated groups 
was introduced in \cite{Grig,Grom} 
and may be constructed as follows. 
Fix $d \in \mathbb{N}$ and an ordered $d$-element set 
$\mathbf{X}=\lbrace x_1,\ldots,x_d \rbrace$. 
We may define $\mathcal{G}_d$ to be the set of all normal 
subgroups of the free group 
$\mathbb{F} = F(\mathbf{X})$ on $\mathbf{X}$. 
Alternatively, the points of $\mathcal{G}_d$ may be described in 
terms of \emph{$d$-markings on groups}: 
if $\Gamma$ is a $d$-generated group and $\mathbf{S} = (s_1,\ldots s_d)$ 
is an ordered generating $d$-tuple for $\Gamma$, 
then the pair $(\Gamma,\mathbf{S})$, 
henceforth to be called a \emph{$d$-marked group}
determines an epimorphism $\pi_{\mathbf{S}} : F(\mathbf{X}) \rightarrow \Gamma$ 
sending $x_i$ to $s_i$ for $1 \leq i \leq d$, and hence the point 
$\ker (\pi_{\mathbf{S}}) \in \mathcal{G}_d$.  
Conversely, every point $N \in \mathcal{G}_d$ determines the 
$d$-generated group $\Gamma_N = F(\mathbf{X})/N$ 
and the generating $d$-tuple $\mathbf{S}_N = (x_i N)_i \in \Gamma ^d$, 
so that $N = \ker (\pi_{\mathbf{S}_N})$. 

By a \emph{quotient of $d$-marked groups} 
$(\Gamma,\mathbf{S}) \twoheadrightarrow (\Delta,\mathbf{T})$ 
we shall mean a homomorphism $\phi : \Gamma \rightarrow \Delta$ 
such that $\phi (s_i) = t_i$ for $1 \leq i \leq d$. 
By the fact that $\mathbf{S}$ and $\mathbf{T}$ generate, 
such $\phi$ is unique (if it exists) and surjective. 
We note that two $d$-marked groups $(\Gamma,\mathbf{S})$ and $(\Delta,\mathbf{T})$ 
determine the same point in $\mathcal{G}_d$ if and only if 
the quotient of $d$-marked groups 
$(\Gamma,\mathbf{S}) \twoheadrightarrow (\Delta,\mathbf{T})$ exists and is an isomorphism. 
We may give $\mathcal{G}_d$ the structure of a metric space, 
as follows. For $N,M \vartriangleleft F(\mathbf{X})$ we write: 
\begin{center}
$\nu ( N,M ) 
= \max\lbrace n \in \mathbb{N} : 
N \cap B_{\mathbf{X}}(n) = M \cap B_{\mathbf{X}}(n) \rbrace \in \mathbb{N} \cup \lbrace \infty \rbrace$. 
\end{center}
and set: 
\begin{center}
$d (N,M) 
= 2^{-\nu ( N,M )}$. 
\end{center}
Then $d$ is a well-defined metric on $\mathcal{G}_d$. 
There is a well-known characterization of the class 
of $d$-generated LEF groups in terms of the topology 
of $\mathcal{G}_d$, as described in the next Lemma, 
which is proved in \cite{VerGor} Section 1.4. 

\begin{lem} \label{MarkedGrpsLEFLemma}
Let $(\Gamma,\mathbf{S})$ be a $d$-marked group 
and let $(\Delta_n,\mathbf{T}_n)$ 
be a sequence of marked finite 
$d$-generated groups. 
Then $(\Delta_n,\mathbf{T}_n)$ converges to 
$(\Gamma,\mathbf{S})$ in $\mathcal{G}_d$ iff, 
for all $r \in \mathbb{N}$, 
for all $n$ sufficiently large 
there is a local embedding 
$\phi_n : B_{\mathbf{S}}(r)\rightarrow \Delta_n$ 
satisfying $\phi_n(s_i)=t_{n,i}$ for $1 \leq i \leq d$. 
In particular, $\Gamma$ is LEF iff for some 
(equivalently any) $d$-marking $\mathbf{S}$ 
of $\Gamma$, $(\Gamma,\mathbf{S})$ lies in the 
closure in $\mathcal{G}_d$ 
of the subspace of marked $d$-generated 
finite groups. 
\end{lem}

\begin{rmrk} \label{MarkedGrpsRmrk}
Let $\Gamma$ be a $d$-generated group. 
If $\Gamma$ is finitely presented then, 
for any $d$-marking $\mathbf{S}$ on $\Gamma$, 
$\ker (\pi_{\mathbf{S}})$ has a finite normal 
generating set, so there exists $C > 0$ such that, 
for any $N \vartriangleleft \mathbf{F}$, 
if $\nu \big( \ker (\pi_{\mathbf{S}}) ,N \big) \geq C$, 
then $\ker (\pi_{\mathbf{S}}) \leq N$. 
It follows that $(\Gamma,\mathbf{S})$ has 
an open neighbourhood in $\mathcal{G}_d$ consisting 
entirely of marked quotients of $(\Gamma,\mathbf{S})$. 
A slightly more general variant, which follows from the same argument, is: 
if $(\Gamma,\mathbf{S})$ and $(\hat{\Gamma},\hat{\mathbf{S}})$ 
are marked $d$-generated groups, 
with $\hat{\Gamma}$ finitely presented, 
and there is a quotient $(\hat{\Gamma},\hat{\mathbf{S}}) \twoheadrightarrow (\Gamma,\mathbf{S})$ of $d$-marked groups, 
then $(\Gamma,\mathbf{S})$ has an open neighbourhood in $\mathcal{G}_d$ consisting 
entirely of marked quotients of $(\hat{\Gamma},\hat{\mathbf{S}})$. 
\end{rmrk}

\begin{defn} \label{LiftDefn}
Given a sequence $(\Gamma_m)_m$ of $d$-generated 
groups, and given for each $m$ a $d$-marking 
$\mathbf{S}_m = (s_{m,1},\ldots,s_{m,d})$ 
on $\Gamma_m$, let: 
\begin{equation*}
\tilde{\mathbf{S}} = (\tilde{s}_1 , \ldots , \tilde{s}_d) \in \big( \prod_m \Gamma_m \big)^d
\end{equation*}
be given by 
$\tilde{s}_{i,m} = s_{m,i} \in \Gamma_m$. 
Let $\otimes (\Gamma_m,\mathbf{S}_m)$ be the subgroup of $\prod_m \Gamma_m$ 
generated by the set $\lbrace \tilde{s}_i : 1 \leq i \leq d \rbrace$ 
(so that $\tilde{\mathbf{S}}$ is a $d$-marking 
on $\otimes (\Gamma_m,\mathbf{S}_m)$). 
We shall refer to the group $\otimes (\Gamma_m,\mathbf{S}_m)$ 
as the \emph{diagonal product} of the sequence $(\Gamma_m,\mathbf{S}_m)_m$, 
and $\tilde{\mathbf{S}}$ as the \emph{diagonal $d$-marking}. 
\end{defn}

Note that the projection $p_n : \prod_m \Gamma_m \rightarrow \Gamma_n$ 
restricts to a quotient of $d$-marked groups 
$\big( \otimes (\Gamma_m,\mathbf{S}_m),\tilde{\mathbf{S}} \big) 
\twoheadrightarrow (\Gamma_n,\mathbf{S}_n)$. 

\begin{propn}[\cite{KaPa} Lemma 4.6] \label{KAPaLiftProp}
Suppose the sequence $(\Gamma_m,\mathbf{S}_m)_m$ 
converges in $\mathcal{G}_d$ to $(\Gamma,\mathbf{S})$. 
Then there is a quotient of $d$-marked groups: 
\begin{equation*} 
\tau : \big( \otimes (\Gamma_m,\mathbf{S}_m) , \tilde{\mathbf{S}} \big) \twoheadrightarrow (\Gamma, \mathbf{S}), 
\end{equation*}
called the \emph{tail homomorphism}, 
with kernel: 
\begin{equation*}
\ker(\tau) = \big( \otimes (\Gamma_m,\mathbf{S}_m) \big) \cap \big( \bigoplus_m \Gamma_m \big)
\end{equation*}
\end{propn}

Since the class of amenable groups is closed under subgroups, 
extensions and ascending unions, we deduce the following, 
which will be used in the next Section. 

\begin{lem} \label{LiftAmenLem}
Let $(\Gamma_m,\mathbf{S}_m)_m$ and $(\Gamma,\mathbf{S})$ 
be as in Proposition \ref{KAPaLiftProp}. 
If $\Gamma$ and the $\Gamma_m$ are amenable, 
then so is $\otimes (\Gamma_m,\mathbf{S}_m)$. 
\end{lem}

The next Proposition gives a sufficient condition 
for local stability of a group $\Gamma$ 
in terms of local stability of a sequence of stable groups 
converging to $\Gamma$ in $\mathcal{G}_d$. 
The hypotheses of the condition are rather strong, 
but they are general enough to allow us to 
exhibit our first explicit example of 
a locally stable finitely generated group which is not stable. 

\begin{propn} \label{liftingalmosthomsprop}
Let $(\Gamma_m,\mathbf{S}_m)_m$ and $(\Gamma,\mathbf{S})$ 
be as in Proposition \ref{KAPaLiftProp}. 
Suppose that for all $m$, there is a quotient of $d$-marked groups 
$\pi^{(m)} : (\Gamma_m,\mathbf{S}_m) \rightarrow (\Gamma,\mathbf{S})$. 
If $\Gamma_m$ is locally stable for all $m$, then $\Gamma$ is locally stable. 
\end{propn}

\begin{proof}
There is a sequence $(r_m)$ in $\mathbb{N}$, tending to $\infty$, 
such that the restriction of $\pi^{(m)}$ to $B_{\mathbf{S}_m} (r_m)$ 
is a bijection onto $B_{\mathbf{S}} (r_m)$. 
Define $\theta ^{(m)} : \Gamma \rightarrow \Gamma_m$ such that, 
for all $h \in B_{\mathbf{S}} (r_m)$, 
$h = (\pi^{(m)} \circ \theta ^{(m)} ) (h)$
(with $\theta ^{(m)} (h)$ defined arbitrarily for 
$h \in \Gamma \setminus B_{\mathbf{S}} (r_m)$). 
Then for all $g,h \in B_{\mathbf{S}} (r_m/2)$, 
$\theta ^{(m)} (gh) = \theta ^{(m)}(g)\theta ^{(m)}(h)$. 


Let $(\phi_n :\Gamma \rightarrow \Sym(k_n))_n$ be an almost-homomorphism. 
Then for any fixed $m$, $(\phi_n \circ \pi^{(m)}) : \Gamma_m \rightarrow \Sym(k_n)$ 
defines an almost-homomorphism, so by local stability, 
there exists a partial homomorphism 
$(\psi ^{(m)} _n : \Gamma_m \rightarrow \Sym(k_n))_n$ such that, 
for all $m$, and all $g \in \Gamma_m$, 
\begin{equation*}
d_{k_n} \big( \psi ^{(m)}(g),(\phi_n \circ \pi^{(m)})(g) \big) \rightarrow 0 \text{ as }m\rightarrow \infty 
\end{equation*}
Fix a sequence $(\epsilon_n)$ of positive reals, 
converging to $0$. 
Then there exists a sequence $(m_n)$ in $\mathbb{N}$, tending to $\infty$, 
such that: 
\begin{itemize}
\item[(i)] $\big( \psi_n ^{(m_n)} \circ \theta ^{(m_n)} : \Gamma \rightarrow\Sym (k_n) \big)_n$ is a partial homomorphism; 

\item[(ii)] For all $n \in \mathbb{N}$ and $g \in B_{\mathbf{S}_{m_n}} (1/\epsilon_n)$, 
\begin{equation*}
d_{k_n} \big( \psi ^{(m_n)}(g),(\phi_n \circ \pi^{(m_n)})(g) \big) < \epsilon_n. 
\end{equation*}
\end{itemize}
Now given $h\in \Gamma$, let $l > 0$ be such that $h \in B_{\mathbf{S}}(l)$. 
Then for all $n$ sufficiently large, $\min (1/\epsilon_n , r_{m_n} ) > l$ 
and there exists $g \in B_{\mathbf{S}_{m_n}} (l)$ with $h = \pi^{(m_n)}(g)$, 
so that $(\psi_n ^{(m_n)} \circ \theta ^{(m_n)})(h) = \psi_n ^{(m_n)} (g)$, and: 
\begin{align*}
d_{k_n} \big( \psi_n ^{(m_n)} \circ \theta ^{(m_n)})(h),\phi_n(h)  \big)
 = d_{k_n} \big( \psi ^{(m_n)}(g),(\phi_n \circ \pi^{(m_n)})(g) \big)< \epsilon_n
\end{align*}
so $\big( \psi_n ^{(m_n)} \circ \theta ^{(m_n)} \big)$ is close 
to $(\phi_n)$, as desired. 
\end{proof}

\begin{rmrk} \label{LimitQuotRmrk}
\normalfont
The hypothesis in Proposition \ref{liftingalmosthomsprop} that 
the convergence of $(\Gamma_m,\mathbf{S}_m)_m$ to $(\Gamma,\mathbf{S})$ 
in $\mathcal{G}_d$ is induced by a sequence of epimorphisms 
$\pi_m : \Gamma_m \rightarrow \Gamma$ cannot be entirely removed. 
For example consider the group described in the proof of Theorem \ref{AbelsGroupThm}: 
it is not locally stable, 
but it is residually finite, hence LEF, 
hence there are \emph{finite} marked groups $(\Gamma_m,\mathbf{S}_m)_m$ 
converging to $(\Gamma,\mathbf{S})$ in $\mathcal{G}_d$, 
and finite groups are stable by \cite{GlRi} Theorem 2. 
\end{rmrk}

Let $\FSym(\mathbb{Z}) \leq \Sym(\mathbb{Z})$ be the group of finitely supported 
permutations of the set $\mathbb{Z}$. 
$\FSym(\mathbb{Z})$ has a subgroup $\FAlt(\mathbb{Z})$ of index two, 
consisting of all finitely supported even permutations. 
Both $\FSym(\mathbb{Z})$ and $\FAlt(\mathbb{Z})$ are normal in $\Sym(\mathbb{Z})$. 
$\FAlt(\mathbb{Z})$ is the ascending union of the alternating groups 
on finite subsets of $\mathbb{Z}$; 
thus $\FAlt(\mathbb{Z})$ is an infinite simple group. 
Let $\rho : \mathbb{Z} \rightarrow \Sym(\mathbb{Z})$ 
denote the regular action of the additive group $\mathbb{Z}$; 
thus $\mathbb{Z}$ acts (through $\rho$) via conjugation on $\Sym(\mathbb{Z})$. 
By normality of $\FAlt(\mathbb{Z})$, we can form the semidirect product 
$\mathcal{A}(\mathbb{Z}) = \FAlt(\mathbb{Z}) \rtimes_{\rho} \mathbb{Z}$, 
called the \emph{alternating enrichment} of $\mathbb{Z}$. 

For $r \in \mathbb{N}$, 
let $\lBrack r \rBrack = \lbrace n \in \mathbb{Z} : \lvert n \rvert \leq r \rbrace$ 
and let $\alpha_r , \beta_r \in \Sym(\lBrack r \rBrack)$ 
be given by $\alpha_r =(-r \; 1-r \cdots r-1 \; r) , \beta_r = (-1 \; 0 \; 1)$. 
Then it is easily seen that $\langle \alpha_r,\beta_r \rangle = \Alt(\lBrack r \rBrack)$, 
and similarly we have the following. 

\begin{lem}
The group $\mathcal{A}(\mathbb{Z})$ is finitely generated by 
$\lbrace (\id_{\mathbb{Z}},1) , ((-1 \; 0 \; 1),0) \rbrace$. 
\end{lem}

We shall prove: 

\begin{thm} \label{AltEnrLocStabThm}
$\mathcal{A}(\mathbb{Z})$ is locally stable but not weakly stable. 
\end{thm}

Let $\alpha_r , \beta_r \in \Alt(\lBrack r \rBrack)$ 
be as above, 
and let $\mathbf{T}(r) = (\alpha_r,\beta_r) \in \Alt ( \lBrack r \rBrack )^2$ 
(a $2$-marking on $\Alt ( \lBrack r \rBrack )$). 
Now let $r : \mathbb{N} \rightarrow \mathbb{N}_{\geq 2}$, with $r$ strictly increasing. 
We consider sequences of the form 
$\big( \Alt ( \lBrack r(n) \rBrack ) , \mathbf{T}(r(n)) \big)$ in $\mathcal{G}_2$. 

\begin{propn} \label{AltEnriLEFProp}
For any function $r : \mathbb{N} \rightarrow \mathbb{N}_{\geq 2}$ as above, 
$\big( \Alt ( \lBrack r(n) \rBrack ) , \mathbf{T} (r(n)) \big)$ converges in $\mathcal{G}_2$ 
to $(\mathcal{A}(\mathbb{Z}),\mathbf{S})$, 
where $\mathbf{S} = \big( (\id_{\mathbb{Z}},1) , ((-1 \; 0 \; 1),0) \big)$. 
\end{propn}

\begin{proof}
This is covered, for example, in Remark 5.4 
of \cite{MimSak}. 
\end{proof}

The groups $G(r) = \otimes \big( \Alt ( \lBrack r(n) \rBrack ) , \mathbf{T}(r(n)) \big)$ 
were originally studied by B.H. Neumann \cite{Neum}, 
who proved that different functions $r$ as above yield nonisomorphic groups $G(r)$. 
Our interest here in these groups stems from the following, 
which is the main result of \cite{LeLu}. 

\begin{thm}
For any increasing function 
$r : \mathbb{N} \rightarrow \mathbb{N}_{\geq 2}$, 
$G(r)$ is stable. 
\end{thm}

Let $\tilde{\mathbf{S}}(r)\in G(r)^2=\otimes \big( \Alt ( \lBrack r(n) \rBrack ),\mathbf{T}(r(n)) \big)^2$ 
be the diagonal $2$-marking described in Definition \ref{LiftDefn}. 
We now fix an increasing function $r_0 : \mathbb{N} \rightarrow \mathbb{N}_{\geq 2}$. 
For $n \in \mathbb{N}$ we define $r_n (m) = r_0 (m+n)$. 
The key to Theorem \ref{AltEnrLocStabThm} is the following observation. 

\begin{propn} \label{AltEnriDLProp}
Write $\mathbf{S}_n = \tilde{\mathbf{S}}(r_n)$. 
The sequence $\big( G(r_n) , \mathbf{S}_n  \big)_n$ converges 
in $\mathcal{G}_2$ to $(\mathcal{A}(\mathbb{Z}),\mathbf{S})$, 
where $\mathbf{S}$ is as in Proposition \ref{AltEnriLEFProp}. 
\end{propn}

\begin{proof}[Proof of Theorem \ref{AltEnrLocStabThm}]
$\FAlt(\mathbb{Z})$ is locally finite, so $\mathcal{A}(\mathbb{Z})$ is amenable 
(being the extension of a locally finite group by an abelian group). 
On the other hand, $\FAlt(\mathbb{Z})$ is an infinite simple group, 
hence $\mathcal{A}(\mathbb{Z})$ is not residually finite. 
By Theorem \ref{ArzPauThm} 
$\mathcal{A}(\mathbb{Z})$ is not weakly stable. 

For fixed $n$, Proposition \ref{AltEnriLEFProp} guarantees that 
$\big( \Alt ( \lBrack r_n(m) \rBrack ) , \mathbf{T}(r_n(m)) \big)_m$ converges in $\mathcal{G}_2$ 
to $(\mathcal{A}(\mathbb{Z}),\mathbf{S})$. 
By Proposition \ref{KAPaLiftProp} we have the tail homomorphism 
$\tau_n : (G(r_n),\mathbf{S}_n) \twoheadrightarrow (\mathcal{A}(\mathbb{Z}),\mathbf{S})$. 
By Proposition \ref{AltEnriDLProp} we see that the sequence $\pi^{(n)}=\tau_n$ 
satisfies the conditions of Proposition \ref{liftingalmosthomsprop}, 
and we conclude that $\mathcal{A}(\mathbb{Z})$ is locally stable. 
\end{proof}

\begin{proof}[Proof of Proposition \ref{AltEnriDLProp}]
Let $K \in \mathbb{N}$, so that by Proposition \ref{AltEnriLEFProp} there exists 
$M \in \mathbb{N}$ be such that, for all $m \geq M$, 
\begin{center}
$d \Big( \big( \Alt ( r_0(m) ) , \mathbf{T}(r_0(m)) \big) , \big( \mathcal{A}(\mathbb{Z}),\mathbf{S}\big) \Big) \leq 2^{-(K+1)}$
\end{center}
so that for any $l,m \geq M$, 
\begin{equation} \label{AltCloseIneq}
d \Big( \big( \Alt ( r_0(l) ) , \mathbf{T}(r_0(l)) \big) , 
\big( \Alt ( r_0(m) ) , \mathbf{T}(r_0(m)) \big) \Big) \leq 2^{-K}
\end{equation}
by the triangle inequality. 
Let $n \geq N$ and claim: 
\begin{equation} \label{LiftConvClaim}
d \Big( \big( G(r_n),\mathbf{S}_n \big) , \big( \mathcal{A}(\mathbb{Z}),\mathbf{S}\big) \Big) \leq 2^{-K}
\end{equation}
which yields the result. Let: 
\begin{center}
$\tau_n : \big( G(r_n),\mathbf{S}_n \big) \twoheadrightarrow \big( \mathcal{A}(\mathbb{Z}),\mathbf{S}\big)$
\end{center}
be the tail homomorphism described in Proposition \ref{KAPaLiftProp}. 
Since $\tau_n \circ \pi_{\mathbf{S}_n} = \pi_{\mathbf{S}}$, 
if (\ref{LiftConvClaim}) fails, 
then there exists a non-trivial reduced word $w \in F(\mathbf{X})$ 
such that $e \neq \pi_{\mathbf{S}_n} (w) \in \ker (\tau_n)$. 
For $k \in \mathbb{N}$ let: 
\begin{center}
$p_k : \big( G(r_n),\mathbf{S}_n \big) \twoheadrightarrow 
\big( \Alt ( r_n(k) ) , \mathbf{T}(r_n(k)) \big)
= \big( \Alt ( r_0(k+n) ) ,\mathbf{T}(r_0(k+n)) \big)$
\end{center}
by projection onto the $k^{th}$ factor. 
By the conclusion of Proposition \ref{KAPaLiftProp}, 
$\pi_{\mathbf{S}_n} (w) \in \bigoplus_k \Alt ( r_0(k+n) )$
so there exist $l,m \in\mathbb{N}$ such that: 
\begin{center}
$\pi_{\mathbf{T}(r(n+l))} (w) = (p_l \circ \pi_{\mathbf{S}_n}) (w) \neq e$ but 
$\pi_{\mathbf{T}(r(n+m))} (w) = (p_m \circ \pi_{\mathbf{S}_n}) (w) = e$. 
\end{center}
It follows that:  
\begin{center}
$d \Big( \big( \Alt ( r_0(n+l) ) , \mathbf{T}(r_0(n+l)) \big) , 
\big( \Alt ( r_0(n+m) ) , \mathbf{T}(r_0(n+m)) \big) \Big) > 2^{-K}$
\end{center}
contradicting (\ref{AltCloseIneq}). 
\end{proof}

\begin{rmrk}
\normalfont
One may also define the \emph{symmetric enrichment} 
$\mathcal{S}(\mathbb{Z}) = \FSym(\mathbb{Z}) \rtimes_{\rho} \mathbb{Z}$, 
in which $\mathcal{A}(\mathbb{Z})$ sits as a subgroup of index $2$. 
We do not know whether local stability is preserved under commensurability 
in general, so one cannot deduce local stability for $\mathcal{S}(\mathbb{Z})$ 
directly from Theorem \ref{AltEnrLocStabThm}. 
That said, one can also prove that $\mathcal{S}(\mathbb{Z})$ is locally stable, 
along the lines of the argument sketched in Example \ref{AltEnrichIRSEx} below. 

Further, much as in Proposition \ref{AltEnriLEFProp}, 
one may construct a sequence of $2$-markings of finite symmetric groups 
converging in $\mathcal{G}_2$ to (a $2$-marking of) $\mathcal{S}(\mathbb{Z})$, 
and thereby produce a family of diagonal product groups admitting 
epimorphisms onto $\mathcal{S}(\mathbb{Z})$. 
We expect that these diagonal product groups are also stable groups, 
and that their stability may be proved using the arguments of \cite{LeLu}. 
\end{rmrk}

\section{Invariant random subgroups} \label{IRSSect}

For $\Gamma$ a countable group, 
let $\Sub(\Gamma)$ be the space of all subgroups of $\Gamma$, 
a closed subspace of the space $\lbrace 0,1 \rbrace^{\Gamma}$ 
of subsets of $\Gamma$ (equipped with the Tychonoff topology). 
The group $\Gamma$ admits a continuous action on $\Sub(\Gamma)$ 
by conjugation, which induces an action on the space $\Prob(\Gamma)$ 
of Borel probability measures on $\Sub(\Gamma)$. 
An \emph{invariant random subgroup (IRS)} of $\Gamma$ is by definition 
a fixed point of the action of $\Gamma$ on $\Prob(\Gamma)$. 

\begin{ex} \label{PtMassEx}
For $H \leq \Gamma$, $\delta_H \in \Prob (\Gamma)$ 
denotes the point mass on the subgroup $H$. 
The measure $\delta_H$ is an IRS iff $H$ is normal 
in $\Gamma$. 
More generally, for $H_1 , \ldots , H_n \leq \Gamma$ 
distinct subgroups 
and $\lambda_1 , \ldots , \lambda_n \in [0,1]$, 
with $\sum_i \lambda_i = 1$, we have a probability 
measure $\sum_i \lambda_i\delta_{H_i}\in\Prob(\Gamma)$. 
The latter is an IRS iff the $H_i$ form 
a union of whole conjugacy classes of subgroups, 
and $\lambda_i=\lambda_j$ whenever 
$H_i$ and $H_j$ are conjugate in $\Gamma$. 
\end{ex}

We denote by $\IRS(\Gamma)$ the set of all IRSs of $\Gamma$; 
it is a compact metrizable space under the weak$^{\ast}$ topology. 
If $\Gamma$ is finitely generated by the set $S$, 
$r \in \mathbb{N}$ and $W \subseteq \Gamma$, we write: 
\begin{center}
$C_{r,W} = \lbrace H \leq \Gamma 
: H \cap B_S (r) = W \cap B_S (r) \rbrace$, 
\end{center}
a clopen subset of $\Sub(\Gamma)$. 
For $\mu , \mu_n \in \IRS (\Gamma)$, 
we have the following useful criterion for convergence 
in the weak$^{\ast}$ topology: 
\begin{equation} \label{W*ConvCrit}
\mu_n \rightarrow \mu \text{ in } \IRS(\Gamma) \text{ iff } \mu_n (C_{r,W}) \rightarrow \mu (C_{r,W}) 
\text{ for all } r \in \mathbb{N} \text{ and } W \subseteq B_S (r).
\end{equation}

It is clear that $\IRS (\Gamma)$ forms a closed 
and convex subspace of $\Prob (\Gamma)$. 
An IRS $\mu$ of $\Gamma$ is \emph{ergodic} if, 
for any $\mu_1 , \mu_2 \in \IRS(\Gamma)$ 
and $t \in (0,1)$, if $\mu = t \mu_1 + (1-t) \mu_2$ 
then $\mu_1 = \mu_2$. 
That is, the ergodic IRSs are precisely the 
extreme points of $\Gamma$. 
Thus, the only closed convex subspace of 
$\IRS(\Gamma)$ containing all ergodic IRSs of $\Gamma$ 
is $\IRS(\Gamma)$ itself. 

As in the previous Section, fix an ordered basis $\mathbf{X}$ 
for the rank-$d$ free group $\mathbb{F}=F(\mathbf{X})$. 
If $\Gamma$ is $d$-generated, then a $d$-marking $\mathbf{S}$ on $\Gamma$ 
induces an embedding of $\IRS(\Gamma)$ into $\IRS (\mathbb{F})$: 
an IRS of $\mathbb{F}$ lies in $\IRS(\Gamma)$ iff 
it is supported on subgroups of $\mathbb{F}$ containing 
$\ker (\pi_{\mathbf{S}} :\mathbb{F}\rightarrow\Gamma)$, where as above, $\pi_{\mathbf{S}}$ denotes 
the epimorphism of $d$-marked groups 
$(\mathbb{F},\mathbf{X}) \twoheadrightarrow (\Gamma,\mathbf{S})$. 
We write $\IRS(\Gamma,\mathbf{S})$ to specify this embedded copy of $\IRS(\Gamma)$ 
inside $\IRS (\mathbb{F})$. 
An IRS of $\Gamma$ is called \emph{finite-index} 
if it is an atomic probability measure 
supported on finite-index subgroups of $\Gamma$, 
and is called \emph{cosofic} in $\Gamma$ if it is 
the weak$^{\ast}$-limit of finite-index IRSs of $\Gamma$. 
Note that if 
$\mu \in \IRS(\Gamma,\mathbf{S}) 
\subseteq \IRS(\mathbb{F},\mathbf{X})$ 
then $\mu$ may a priori be cosofic in $\mathbb{F}$ 
but not cosofic in $\Gamma$. 

For $X$ any $\Gamma$-set, 
there is a map $\Stab : X \rightarrow \Sub(\Gamma)$ 
sending a point $x \in X$ to its stabilizer in $\Gamma$. 
If $X$ is a standard Borel space, 
and the action of $\Gamma$ on $X$ is Borel, 
then $\Stab$ is a Borel map. 
Thus, given a Borel regular 
probability measure $\mu$ on $X$, 
there is a pushforward measure 
$\Stab_{\ast}(\mu) \in \Prob (\Gamma)$ 
on $\Sub(\Gamma)$. 
If $\mu$ is $\Gamma$-invariant 
(that is, if the action of $\Gamma$ on $(X,\mu)$ 
is a pmp action), 
then $\Stab_{\ast}(\mu)$ is an IRS of $\Gamma$. 
Although we shall not use the fact in the sequel, 
it in fact transpires that \emph{every} IRS 
can be produced in this way 
(see Proposition 1.4 of \cite{AbGlVi}). 
In the special case that $X$ is a finite discrete 
$\Gamma$-set, the IRS induced in this way 
is a finitely-supported finite-index IRS of $\Gamma$. 

\begin{defn}
Let $X$ be a finite $\Gamma$-set, 
and let $\nu$ be the 
uniform probability measure on $X$ 
(so that $\nu$ is $\Gamma$-invariant). 
By the \emph{IRS associated to $X$} 
we shall mean $\Stab_{\ast}(\nu) \in \IRS(\Gamma)$. 
We call the sequence $(X_n)$ of finite 
$\mathbb{F}$-sets \emph{convergent} if 
the sequence of IRSs associated to the $X_n$ 
converges in $\IRS (\mathbb{F})$. 
\end{defn}

\begin{rmrk} \label{BiggerSetRmrk}
If $X$ and $Y$ are finite $\Gamma$-sets 
with associated IRSs $\mu$ and $\nu$, respectively, 
then $X \sqcup Y$ is a finite $\Gamma$-set 
with associated IRS: 
\begin{equation*}
(\frac{\lvert X \rvert}{\lvert X \rvert + \lvert Y \rvert}) \mu 
+ (\frac{\lvert Y \rvert}{\lvert X \rvert + \lvert Y \rvert}) \nu. 
\end{equation*}
Repeatedly applying this observation, 
we have that for any finite $\Gamma$-set $X$ 
and any $N > 0$, there is a finite 
$\Gamma$-set $X'$ with the same associated IRS as $X$, 
and $\lvert X' \rvert \geq N$. 
\end{rmrk}

Generalizing Remark \ref{BiggerSetRmrk}, 
we have the following slight variation of 
Lemma 7.6 of \cite{BeLuTh}, 
which will be needed in the proof of Theorem 
\ref{IRSMainThm} below. 

\begin{lem} \label{BiggerSetLem}
Let $(\Delta_n , \mathbf{T}_n)$ be a sequence of 
marked $d$-generated groups and 
let $Y_n$ be a finite $\Delta_n$-set, 
with associated IRS $\nu_n$. 
Suppose that $\lvert Y_n \rvert \rightarrow \infty$ 
as $n \rightarrow \infty$ and 
that $(\nu_n)$ converges to 
$\nu \in \IRS(\mathbb{F},\mathbf{X})$. 
Then for any sequence $(m_k)$ of positive integers 
satisfying $m_k \rightarrow \infty$ 
as $k \rightarrow \infty$, 
there exists an unbounded nondecreasing 
function $f : \mathbb{N} \rightarrow \mathbb{N}$ 
and finite $\Delta_{f(k)}$-sets $Y_k '$ 
with associated IRSs $\nu_k '$, 
such that $\lvert Y_k ' \rvert = m_k$ for all $k$, 
and $(\nu_k ')$ converges to $\nu$ also. 
\end{lem}

\begin{proof}
For $r$ a positive integer, 
let $Z_r$ be a set of size $r$, 
on which we let $\mathbb{F}$ act trivially. 
Then for any quotient $\Gamma$ of $\mathbb{F}$, 
$Z_r$ is naturally a $\Gamma$-set. 
Let $(i_n)$ be a strictly increasing sequence 
of positive integers such that 
for all $i_n \leq k < i_{n+1}$, 
$\lvert Y_n \rvert / m_k < 1/n$. 
Set $f (k) = n$. 
For each $i_n \leq k < i_{n+1}$ 
write $m_k = q_k \lvert Y_n \rvert + r_k$, 
with $q_k \geq n$ and 
$0 \leq r_k < \lvert Y_n \rvert$, and set: 
\begin{equation*}
Y_k ' = Z_{r_k} \sqcup \bigsqcup_{q_k} Y_n
\end{equation*}
a finite $\Delta_n$-set with $\lvert Y_k ' \rvert=m_k$. 
By Remark \ref{BiggerSetRmrk}, 
the IRS $\nu_k ' \in \IRS (\Delta_n , \mathbf{T}_n)$ 
associated to $Y_k '$ is given by: 
\begin{equation*}
\nu_k ' = \frac{m_k - r_k}{m_k} \nu_n + \frac{r_k}{m_k} \delta_{\Delta_n}
\end{equation*}
Then $(\nu_k ')$ converges to $\nu$, 
since $r_k / m_k < \lvert Y_n \rvert/m_k < 1/n \rightarrow 0$ as $k \rightarrow \infty$, 
because $k < i_{n+1}$. 
\end{proof}

The IRSs associated to finite $\Gamma$-sets are 
one source of examples of finite-index IRSs of $\Gamma$. 
Not every finite-index IRS is necessarily of this form; 
for instance if $\mu \in \IRS(\Gamma)$ 
is the IRS associated to a finite $\Gamma$-set, 
then for every $r \in \mathbb{N}$ 
and $W \subseteq \Gamma$, 
each $\mu(C_{r,W})$ is a rational number. 
Nevertheless, the next construction 
(which is a slight modification of 
Lemma 4.4 of \cite{BeLuTh}) 
shows that every finite-index IRS may be \emph{approximated} 
by IRSs associated to actions on finite sets. 

\begin{lem} \label{IRSApproxLem}
Let $(\Gamma_n , \mathbf{S}_n)$ be a sequence of 
marked $d$-generated groups; 
let $\mu_n \in \IRS (\Gamma_n , \mathbf{S}_n)$ 
be a finite-index IRS, and suppose 
the sequence $(\mu_n)$ converges to some 
$\mu \in \IRS(\mathbb{F},\mathbf{X})$. 
Then for each $n$ there exists a finite 
$\Gamma_n$-set $X_n$ such that, 
letting $\nu_n \in \IRS(\Gamma_n , \mathbf{S}_n)$ 
be the IRS associated to $X_n$, 
the sequence $(\nu_n)$ converges to 
$\mu$ also. 
\end{lem}

\begin{proof}


Being finite-index, $\mu_n$ is in particular 
a cosofic IRS, so Lemma 4.4 of \cite{BeLuTh} applies. 
There is a sequence $(X_{n,m})_m$ of finite 
$\Gamma_n$-sets whose associated sequence 
$(\mu_{n,m})_m$ of IRSs converges to $\mu_n$. 
Since $(\mu_n)$ converges to $\mu$, 
there exists increasing 
$f:\mathbb{N}\rightarrow\mathbb{N}$ 
such that $(\mu_{n,f(n)})$ converges to $\mu$. 
We may therefore take $X_n = X_{n,f(n)}$. 
\end{proof}

In \cite{BeLuTh} (Lemma 7.5), the following is proved. 

\begin{lem} \label{BLT75Lem}
Let $(X_n)$ be a convergent sequence of finite $\mathbb{F}$-sets, 
with $\mu \in \IRS(\mathbb{F},\mathbf{X})$ being the 
limit of the sequence $(\Stab_{\ast}(\nu_n))$ of IRS 
associated to the $X_n$. 
Then $(X_n)$ is a stability challenge for $\Gamma$ 
iff $\mu$ lies in $\IRS(\Gamma,\mathbf{S})$. 
\end{lem}

Moreover it is shown that in proving stability 
of $\Gamma$, it suffices to consider 
\emph{convergent} stability challenges. 
In a similar vein we have the following. 

\begin{propn} \label{ConvStabChallProp}
A finitely generated group $\Gamma$ is locally stable iff 
every convergent stability challenge for $\Gamma$ has a local solution. 
\end{propn}

\begin{proof}
Let $\mathbf{S}$ be a $d$-marking on $\Gamma$. 
By Proposition \ref{StabChallProp}, it suffices to prove that if 
every convergent stability challenge has a local solution, 
then so does every stability challenge. 
Suppose to the contrary, that there is a stability challenge $(X_n)$ 
for $\Gamma$ with no local solution. 
For $m$ a positive integer and $X$ a finite $\mathbb{F}$-set, 
we shall describe a finite $\mathbb{F}$-set $Y$ as ``$m$-good' for $X$ if: 
$\lvert X \rvert = \lvert Y \rvert$; 
$d_{\gen}(X,Y) < 1/m$ and $\ker (\pi_{\mathbf{S}}) \cap B_{\mathbf{X}} (m)$ 
acts trivially on $Y$. 
In this terminology, a sequence $(Y_n)$ of finite $\mathbf{F}$-sets 
is a local solution for $(X_n)$ iff for all $m$, 
$Y_n$ is $m$-good for $X_n$ for all but finitely many $n$. 
Therefore, passing to a subsequence of our $(X_n)$, 
we may assume that there exists a positive integer $m$ such 
that for no $n$ does there exist a finite $\mathbb{F}$-set which 
is $m$-good for $X_n$. 
This $(X_n)$ is then a stability challenge for $\Gamma$, 
no subsequence of which has a local solution. 
On the other hand, by compactness of $\IRS(\mathbb{F})$, 
$(X_n)$ has a subsequence which is a convergent stability challenge for $\Gamma$, 
contradiction. 
\end{proof}

We now come to the main result of this Section, 
which is our necessary and sufficient 
condition for local stability 
of a finitely generated amenable group 
in terms of IRSs. 

\begin{defn}
Let $\Gamma$ be a $d$-generated group; 
let $\mathbf{S}$ be a $d$-marking on $\Gamma$, 
and let $\mu \in \IRS (\Gamma,\mathbf{S})$. 
We call $\mu$ \emph{partially cosofic in 
$(\Gamma,\mathbf{S})$} 
if there exists a sequence of $d$-marked finite groups $(\Delta_n,\mathbf{T}_n)$ 
converging to $(\Gamma,\mathbf{S})$ in $\mathcal{G}_d$ 
and $\nu_n \in \IRS (\Delta_n,\mathbf{T}_n)$ such that 
$\nu_n \rightarrow \mu$ in $\IRS(\mathbb{F},\mathbf{X})$. 
\end{defn}

\begin{thm} \label{IRSMainThm}
Suppose $\Gamma$ is a $d$-generated amenable group 
and let $\mathbf{S}$ be a $d$-marking on $\Gamma$. 
Then $\Gamma$ is locally stable if and only if 
every $\mu \in \IRS (\Gamma,\mathbf{S})$ is 
partially cosofic in $(\Gamma,\mathbf{S})$. 
\end{thm}

\begin{rmrk} \label{TrivialIRSRmrk}
The $d$-generated group $\Gamma$ is LEF iff for some (equivalently any) 
$d$-marking $\mathbf{S}$ on $\Gamma$ there exists a sequence of 
$d$-marked finite groups $(\Delta_n,\mathbf{T}_n)$ 
converging to $(\Gamma,\mathbf{S})$ in $\mathcal{G}_d$. 
For \emph{any} such sequence, the trivial IRS 
$\delta_{\lbrace e \rbrace}$ and 
$\delta_{\Delta_n} \in \IRS (\Delta_n,\mathbf{T}_n)$ 
converge, respectively, to the trivial IRS 
$\delta_{\lbrace e \rbrace}$ 
and $\delta_{\Gamma} \in \IRS (\Gamma,\mathbf{S})$, 
so these IRSs are partially cosofic in any LEF group. 
\end{rmrk}

Before embarking on the proof of Theorem 
\ref{IRSMainThm}, we note one slight refinement, 
which will be useful in applications, 
particularly in the next Section. 

\begin{coroll} \label{IRSMainCoroll}
Let $(\Gamma,\mathbf{S})$ 
be as in Theorem \ref{IRSMainThm}. 
The following are equivalent: 
\begin{itemize}
\item[(i)] $\Gamma$ is locally stable; 

\item[(ii)] Every $\mu \in \IRS (\Gamma,\mathbf{S})$ is 
partially cosofic in $(\Gamma,\mathbf{S})$; 

\item[(iii)] Every ergodic 
$\mu \in \IRS (\Gamma,\mathbf{S})$ is 
partially cosofic in $(\Gamma,\mathbf{S})$. 
\end{itemize}
\end{coroll}

\begin{proof}
Given Theorem \ref{IRSMainThm}, 
the only nontrivial implication is from (iii) to (ii). 
Let $\mathcal{P} \subseteq \IRS (\Gamma,\mathbf{S})$ 
be the set of partially cosofic IRSs in 
$(\Gamma,\mathbf{S})$. 
As noted when we first defined ergodic IRSs, 
it suffices to show that 
$\mathcal{P}$ is closed and convex. 
For closure, suppose that $(\mu_n)$ is a sequence 
consisting of partially cosofic IRSs in 
$(\Gamma,\mathbf{S})$, 
and converging to $\mu \in \IRS (\Gamma,\mathbf{S})$. 
For each $n$, let $(\Delta_m ^{(n)},\mathbf{T}_m ^{(n)})$ be finite $d$-marked groups and 
$\nu_m ^{(n)} \in \IRS (\Delta_m ^{(n)},\mathbf{T}_m ^{(n)})$ witness the partial cosoficity of 
$\mu_n$ in $(\Gamma,\mathbf{S})$. 
Then for some increasing 
$f: \mathbb{N} \rightarrow \mathbb{N}$, 
the sequences $(\Delta_{f(n)} ^{(n)},\mathbf{T}_{f(n)} ^{(n)})$ and $\nu_{f(n)} ^{(n)}$ witness 
that $\mu$ is partially cosofic in 
$(\Gamma,\mathbf{S})$ also. 

For convexity, let $\mu_1 , \mu_2 \in \mathcal{P}$ 
and $t \in (0,1)$. 
For $i=1,2$, let $(\Delta_n ^{(i)},\mathbf{T}_n ^{(i)})$  
and $\nu_n ^{(i)} \in \IRS (\Delta_n ^{(i)},\mathbf{T}_n ^{(i)})$ witness the partial cosoficity of $\mu_i$. 
It is clear that 
$\nu_n = t \nu_n ^{(1)} + (1-t) \nu_n ^{(2)}$ 
converges to $t \mu_1 + (1-t) \mu_2$ in 
$\IRS(\mathbb{F},\mathbf{X})$. 
It therefore suffices to find an associated sequence 
of finite groups. 

For $1 \leq j \leq d$, 
let $t_{n,j}  = (t_{n,j} ^{(1)}, t_{n,j} ^{(2)})$, 
where $\mathbf{T}_n ^{(i)} 
= (t_{n,1} ^{(i)},\ldots , t_{n,d} ^{(i)})$, 
and let: 
\begin{equation*}
\Delta_n = \langle t_{n,1} , \ldots , t_{n,d} \rangle 
\leq \Delta_n ^{(1)} \times \Delta_n ^{(2)}
\end{equation*}
(so that the projection of $\Delta_n$ to each factor is surjective). 
It is easy to see that $(\Delta_n,\mathbf{T}_n)$ 
converges to $(\Gamma,\mathbf{S})$ in 
$\mathcal{G}_d$, since the $(\Delta_n ^{(i)},\mathbf{T}_n ^{(i)})$ do. 
The $\nu_n ^{(i)}$ are finitely supported atomic IRSs, 
hence $\nu_n$ is too. 
Finally, if $H \leq \mathbb{F}$ lies in the support of $\nu_n$, 
then for one of $i = 1$ or $2$, $H$ lies in the 
support of $\nu_n ^{(i)}$, 
so $\ker(\pi_{\mathbf{T}}) \leq \ker(\pi_{\mathbf{T}^{(i)}}) \leq H$, 
and $\nu_n \in \IRS (\Delta_n,\mathbf{T}_n)$, 
as required. 
\end{proof}

There is defined in Section 6 of \cite{BeLuTh} a notion 
of ``statistical distance'' $d_{\stat}(X,Y)$ 
between a pair $X$ and $Y$ of Borel probabiliy 
spaces with Borel $\mathbb{F}$-actions. 
We will not need the full definition of $d_{\stat}$; 
only the following consequences. 

\begin{propn} \label{dstatProp}
Let $(X_n)$ and $(Y_n)$ be sequences of finite 
$\mathbb{F}$-sets, 
and let $\mu_n$ and 
$\nu_n \in \IRS (\mathbb{F},\mathbf{X})$ 
be the IRSs associated to $X_n$ and $Y_n$, 
respectively. 
\begin{itemize}
\item[(i)] If there exists 
$\lambda \in \IRS (\mathbb{F},\mathbf{X})$ 
such that $\mu_n \rightarrow \lambda$ 
and $\nu_n \rightarrow \lambda$, 
then we have that $d_{\stat} (X_n,Y_n) \rightarrow 0$; 

\item[(ii)] If $\mu_n \rightarrow \lambda \in \IRS (\mathbb{F},\mathbf{X})$ 
and $d_{\stat} (X_n,Y_n) \rightarrow 0$ then $\nu_n \rightarrow \lambda$ also; 

\item[(iii)] If $\lvert X_n \rvert = \lvert Y_n \rvert$ for all $n$ 
and $d_{\gen} (X_n,Y_n) \rightarrow 0$ 
then $d_{\stat} (X_n,Y_n) \rightarrow 0$; 

\item[(iv)] Suppose there exists a $d$-marked amenable 
group $(\Gamma,\mathbf{S})$ such that for all $n$, 
the action of $\mathbb{F}$ on $Y_n$ factors through 
$\pi_{\mathbf{S}}$, and 
$\lvert X_n \rvert = \lvert Y_n \rvert$ for all $n$. 
If we have $d_{\stat} (X_n,Y_n) \rightarrow 0$, 
then $d_{\gen} (X_n,Y_n) \rightarrow 0$. 

\end{itemize}
\end{propn}

\begin{proof}
Items (i) and (ii) follow from Lemma 6.1; 
item (iii) is Proposition 6.3, 
and item (iv) is the content of Proposition 6.8, 
all from \cite{BeLuTh}.  
\end{proof}

\begin{proof}[Proof of Theorem \ref{IRSMainThm}]
Suppose that the hypothesis on the IRSs holds. 
Let $(X_n)$ be a stability challenge for $\Gamma$, 
which by Proposition \ref{ConvStabChallProp} we can 
assume to be a convergent stability challenge. 
Let $\mu_n \in \IRS(\mathbb{F},\mathbf{X})$ be the 
IRS associated to $X_n$, 
so that there exists $\mu \in \IRS (\Gamma,\mathbf{S})$ 
with $\mu_n \rightarrow \mu$ in 
$\IRS(\mathbb{F},\mathbf{X})$. 
Let $(\Delta_n,\mathbf{T}_n)$ and $\nu_n \in \IRS (\Delta_n,\mathbf{T}_n)$ 
witness the partial cosoficity of $\mu$ 
in $(\Gamma,\mathbf{S})$. 
By Lemmas \ref{BiggerSetLem} and \ref{IRSApproxLem}, 
we may assume that there exists a finite $\Delta_n$-set $Y_n$, 
with $\lvert X_n\rvert = \lvert Y_n \rvert$, whose associated IRS is $\nu_n$. 
We have $\mu_n ,\nu_n \rightarrow \mu$, 
so $d_{\stat}(X_n,Y_n) \rightarrow 0$ 
by Proposition \ref{dstatProp} (i). 

Let $(\tilde{\Gamma},\mathbf{T}) = \bigotimes (\Delta_n,\mathbf{T}_n)$. 
Then we have epimorphisms of marked groups 
$t:(\tilde{\Gamma},\mathbf{T}) \twoheadrightarrow (\Gamma,\mathbf{S})$ 
and $p_n : (\tilde{\Gamma},\mathbf{T}) \twoheadrightarrow (\Delta_n,\mathbf{T}_n)$, 
so that the $Y_n$ are finite $\tilde{\Gamma}$-sets. 
Moreover $\tilde{\Gamma}$ is amenable 
by Lemma \ref{LiftAmenLem}, 
so by Proposition \ref{dstatProp} (ii) 
(applied to $(\tilde{\Gamma},\mathbf{T})$ 
instead of $(\Gamma,\mathbf{S})$), 
we conclude $d_{\gen}(X_n,Y_n) \rightarrow 0$. 
Viewing $X_n$ and $Y_n$ as finite $\mathbb{F}$-sets, 
$Y_n$ is thus a solution to the stability challenge 
$(X_n)$ for $\mathbb{F}$. 
Finally, since the action of $\mathbb{F}$ on $Y_n$ 
factors through the marked quotient 
$(\mathbb{F},\mathbf{X}) \twoheadrightarrow (\Delta_n,\mathbf{T}_n)$, 
and $(\Delta_n,\mathbf{T}_n)$ 
converges to $(\Gamma,\mathbf{S})$ in $\mathcal{G}_d$, 
every $r \in \ker (\pi_{\mathbf{S}})$ satisfies 
$\pi_{\mathbf{T}_n} (r) = e$ 
for all $n$ sufficiently large, 
hence for such $n$, 
$r$ lies in the kernel of the action of $\mathbb{F}$ 
on $Y_n$. 
Thus $(Y_n)$ is a local $\Gamma$-set, 
and hence a local solution for $\Gamma$ 
to the stability challenge $(X_n)$. 
By Proposition \ref{ConvStabChallProp} 
we conclude that $\Gamma$ is locally stable. 

Conversely, suppose that $\Gamma$ is locally stable, 
and let $\mu \in \IRS (\Gamma,\mathbf{S})$. 
Since $\Gamma$ is amenable, 
by \cite{BeLuTh}[Proposition 6.6] there exist finite-index IRSs 
$\mu_n \in \IRS(\mathbb{F},\mathbf{X})$ with $\mu_n \rightarrow \mu$. 
By Lemma \ref{IRSApproxLem}, 
we may assume that there exists a finite 
$\mathbb{F}$-set $X_n$ with associated IRS $\mu_n$. 
By Lemma \ref{BLT75Lem}, $(X_n)$ is a stability challenge for $\Gamma$. 
By local stability, $(X_n)$ has a local solution $(Y_n)$ for $\Gamma$, 
by Proposition \ref{StabChallProp}. 
Now recall that there is a sequence 
of marked finitely presented groups $(\Gamma_n , \mathbf{S}_n)$ 
such that: 
\begin{itemize}
\item[(i)] There exist marked epimorphisms 
$(\Gamma_n , \mathbf{S}_n) \twoheadrightarrow (\Gamma_{n+1} , \mathbf{S}_{n+1})$ 
and $(\Gamma_n , \mathbf{S}_n) \twoheadrightarrow (\Gamma,\mathbf{S})$; 

\item[(ii)] $(\Gamma_n , \mathbf{S}_n)$ converges to $(\Gamma,\mathbf{S})$ 
in $\mathcal{G}_d$. 

\end{itemize}

Specifically, let $(r_n)$ be an increasing sequence 
of positive integers; let $N_n \in \mathcal{G}_d$ 
be the normal closure in $\mathbb{F}$ 
of $B_{\mathbf{X}} (r_n) \cap \ker (\pi_{\mathbf{S}})$, 
and let $(\Gamma_n,\mathbf{S}_n)$ be the $d$-marked 
group associated to $N_n$. 

Passing to subsequences, we may assume that $Y_n$ is a 
$\Gamma_n$-set by Remark \ref{MarkedGrpsRmrk} (in its more general version). 
Let $\mu_n ' \in \IRS(\Gamma_n , \mathbf{S}_n)$ be the IRS associated to $Y_n$ 
(a finitely supported finite-index IRS of $\Gamma_n$). 
Then $\mu_n ' \rightarrow \mu$ in $\IRS(\mathbb{F},\mathbf{X})$, 
by Proposition \ref{dstatProp} (ii) and (iii). 
Now, since $\Gamma$ is amenable locally stable, it is LEF by Lemma \ref{amenwlslem}, 
so there exists a sequence $(\Lambda_n,\mathbf{U}_n)$ 
of marked finite groups converging to $(\Gamma,\mathbf{S})$ in $\mathcal{G}_d$. 
Passing to subsequences, we may assume 
(again, by the more general form of Remark \ref{MarkedGrpsRmrk}) 
that there exists a marked epimorphism 
$\rho_n : (\Gamma_n , \mathbf{S}_n) \twoheadrightarrow (\Lambda_n,\mathbf{U}_n)$, 
and that:
\begin{equation} \label{BallIsoEqn}
d \big( (\Gamma,\mathbf{S}) \big),(\Gamma_n,\mathbf{S}_n) \big) , 
d \big( (\Gamma,\mathbf{S}) \big),(\Lambda_n,\mathbf{U}_n) \big) , 
d \big( (\Gamma_n,\mathbf{S}_n),(\Lambda_n,\mathbf{U}_n) \big) \leq 2^{-n}
\end{equation}
in $\mathcal{G}_d$. 
Write: 
\begin{equation*}
\mu_n ' = \sum_{i=1} ^{M_n} \lambda_{n,i} \delta_{H_{n,i}}
\end{equation*}
for some finite-index subgroups $H_{n,i}$ of $\Gamma_n$. 
Let $K_{n,i}$ be the normal core of $H_{n,i}$ in $\Gamma_n$, and let: 
\begin{equation*}
K_n = \bigcap_{i=1} ^{M_n} K_{n,i}. 
\end{equation*}
Let $\pi_n : \Gamma_n \twoheadrightarrow \Lambda_n \times (\Gamma_n/K_n)$ 
be given by $\pi_n (g) = (\rho_n (g),gK_n)$; 
let $\mathbf{T}_n = \pi_n (\mathbf{S}_n)$ and 
set $\Delta_n = \im (\pi_n) \leq \Lambda_n \times (\Gamma_n/K_n)$, 
so that $\pi_n : (\Gamma_n , \mathbf{S}_n) \twoheadrightarrow (\Delta_n , \mathbf{T}_n)$ 
is a quotient of $d$-marked groups. 
Projection to the first factor of 
$\Lambda_n \times (\Gamma_n/K_n)$ 
induces a quotient of $d$-marked groups 
$(\Delta_n , \mathbf{T}_n) 
\twoheadrightarrow (\Lambda_n,\mathbf{U}_n)$. 
Since $\rho_n$ restricts to an isomorphism 
of balls of radius $n$, and $\rho_n$ factors through 
$(\Delta_n , \mathbf{T}_n)$, 
it follows from (\ref{BallIsoEqn}) that: 
\begin{equation} \label{BallIsoEqn2}
d\big( (\Gamma , \mathbf{S}),(\Delta_n , \mathbf{T}_n) \big) \leq 2^{-n}
\end{equation}
and $(\Delta_n , \mathbf{T}_n)$ converges in 
$\mathcal{G}_d$ to $(\Gamma , \mathbf{S})$. 
Define:
\begin{equation*}
\nu_n = \sum_{i=1} ^{M_n} \lambda_{n,i} \delta_{\pi_n (H_{n,i})} 
\in \IRS (\Delta_n,\mathbf{T}_n). 
\end{equation*}
Note that $\nu_n$ is indeed an IRS: 
since $\mu_n '$ is an IRS, 
the subgroups $H_{n,i}$ and the coefficients 
$\lambda_{n,i}$ satisfy the criterion 
described in Example \ref{PtMassEx}, 
hence so too do the $\pi_n (H_{n,i})$. 
We claim that $\nu_n \rightarrow \mu$ in $\IRS(\mathbb{F},\mathbf{X})$, 
which will complete the proof. 
It suffices to show that for all $r \in \mathbb{N}$, 
\begin{center}
$B_{\mathbf{X}} (r) \cap \pi_{\mathbf{S}_n} ^{-1} (H_{n,i}) 
= B_{\mathbf{X}} (r) \cap \pi_{\mathbf{T}_n} ^{-1} (\pi_n (H_{n,i})) $
\end{center}
for all $n$ sufficiently large, 
as in that case we have for all $r \in \mathbb{N}$ 
and $W \subseteq B_{\mathbf{X}} (r)$, 
\begin{center}
$\delta_{H_{n,i}} (C_{r,W}) = \delta_{\pi_n (H_{n,i})} (C_{r,W})$
\end{center}
for all $n$ sufficiently large, 
and since $\mu_n ' \rightarrow \mu$ 
the criterion (\ref{W*ConvCrit}) above applies. 

Since $\pi_{\mathbf{T}_n} = \pi_n \circ \pi_{\mathbf{S}_n}$, we have $\pi_{\mathbf{S}_n} ^{-1} (H_{n,i}) \subseteq \pi_{\mathbf{T}_n} ^{-1} (\pi_n (H_{n,i}))$. 
For the converse inclusion, 
let $w \in B_{\mathbf{X}} (r)$ and $h \in H_{n,i}$ be such that 
$\pi_n (h) = \pi_{\mathbf{T}_n}(w)$. 
Then there exists $v \in B_{\mathbf{X}} (r)$ such that 
$\pi_{\mathbf{S}_n} (v) = h$, 
so $v w^{-1}\in\ker (\pi_{\mathbf{T}_n})\cap B_{\mathbf{X}} (2r)$. 
For $n > 2r$, 
and by (\ref{BallIsoEqn}) and (\ref{BallIsoEqn2}), 
$\pi_{\mathbf{S}_n} (w) = \pi_{\mathbf{S}_n} (v) = h$, 
as desired. 
\end{proof}

\begin{ex} \label{AltEnrichIRSEx}
\normalfont
Since the group $\mathcal{A}(\mathbb{Z})$ is amenable, given Theorem \ref{AltEnrLocStabThm} 
we may apply Theorem \ref{IRSMainThm} 
to deduce that every IRS of $\mathcal{A}(\mathbb{Z})$ 
is a limit of IRSs of some finite groups 
converging to $\mathcal{A}(\mathbb{Z})$ in 
$\mathcal{G}_2$. 
Alternatively, one could deduce 
Theorem \ref{AltEnrLocStabThm} 
from Theorem \ref{IRSMainThm} by constructing 
such finite groups and their IRSs directly. 
Let us briefly sketch how this may be done. 
The ergodic IRSs of $\mathcal{A}(\mathbb{Z})$ are 
described in \cite{LeLu}[Section 4]. 
They arise either as atomic measures on finite-index 
normal subgroups of $\mathcal{A}(\mathbb{Z})$, 
or are supported on subgroups of $\FAlt(\mathbb{Z})$. 
By a result of Vershik, the latter class of IRSs 
arise as the stabilizers of a random colouring 
of the integers. That is to say, 
they are pushforwards $\nu_{\alpha}$ 
under the stabilizer map 
of the $\FAlt(\mathbb{Z})$-invariant ergodic probability 
measures $\mu_{\alpha}$ on the space of colourings of 
$\mathbb{Z}$, 
according to which 
each integer is independently coloured 
according to the random variable 
$\alpha \in [0,1]^C$, 
for $C$ some countable set of colours. 

Now let $\mathbf{T} (n)$ be the $2$-markings 
of the finite groups $\Alt ( \lBrack n \rBrack )$ 
described in Proposition \ref{AltEnriLEFProp}, 
so that the sequence  
$\big( \Alt ( \lBrack n \rBrack ),\mathbf{T} (n) \big)$ converges in $\mathcal{G}_2$ 
to $(\mathcal{A}(\mathbb{Z}),\mathbf{S})$. 
Given a colour distribution $\alpha$, 
colour each point of $\lBrack n \rBrack$ 
independently according to $\alpha$ 
to obtain an $\Alt ( \lBrack n \rBrack )$-invariant 
probability measure $\mu_{\alpha,n}$ 
on the set of colourings of $\lBrack n \rBrack$. 
Then $\nu_{\alpha,n} = \Stab_{\ast}(\mu_{\alpha,n})$ 
is an IRS of $\Alt (\lBrack n\rBrack )$, 
and the sequence $(\nu_{\alpha,n})_n$ 
converges in $\IRS (\mathbb{F})$ to $\nu_{\alpha}$. 
\end{ex}

\section{Topological full groups of minimal subshifts} 
\label{TFGSect}

Let $X$ be the Cantor space, and let $T : X \rightarrow X$ be a homeomorphism. 
We refer to the pair $(X,T)$ as a \emph{Cantor dynamical system}. 
The system $(X,T)$ is \emph{minimal} if every orbit in $X$ 
under the action of $\langle T \rangle$ is dense in $X$. 
Henceforth assume that $(X,T)$ is a 
minimal Cantor dynamical system. 

\begin{defn}
The \emph{topological full group} $\lBrack T \rBrack$ 
of the Cantor dynamical system $(X,T)$ 
is the set of all homeomorphisms $g$ of $X$ 
such that there exists a continuous function 
$f_g : X \rightarrow \mathbb{Z}$ 
(called the \emph{orbit cocycle} of $g$) such that 
for all $x \in X$, $g (x) = T^{f_g (x)} (x)$ 
(here we assume $\mathbb{Z}$ equipped with the discrete topology). 
\end{defn}

Equivalently, $g \in \Homeo (X)$ 
lies in $\lBrack T \rBrack$ if there is 
a finite clopen partition $C_1 , \ldots , C_d$ of $X$ 
and integers $a_1 , \ldots , a_d$ such that for $1 \leq i \leq d$, 
$g|_{C_i} = T^{a_i} |_{C_i}$ 
(taking $\lbrace a_1,\ldots,a_d\rbrace = \im(f_g)$, 
$C_i = f_g ^{-1}(a_i)$). 
It is straightforward to check that $\lBrack T \rBrack$ 
is a subgroup of $\Homeo(X)$. 

\begin{rmrk} \label{OrbitCocycRmrk}
\normalfont
We note some immediate consequences of the definition: 
\begin{itemize}
\item[(i)] The orbit cocycle $f_g$ is uniquely determined by 
$g \in \lBrack T \rBrack$, 
since, by minimality, $T$ has no finite orbits on $X$. 
\item[(ii)] For $g,h \in \lBrack T \rBrack$ 
and $x \in X$, we have the cocycle relation: 
\begin{equation*}
f_{gh} (x) = f_g (h(x)) + f_h (x). 
\end{equation*}
\end{itemize}
\end{rmrk}

The group $\lBrack T \rBrack$ and its derived subgroup $\lBrack T \rBrack^{\prime}$ 
have a remarkable collection of group-theoretic 
properties. 

\begin{thm}[Theorems 4.9 and 5.4 of \cite{Mat}] \label{FullGrpSimpleThm}
For any $(X,T)$ as above, 
$\lBrack T \rBrack^{\prime}$ is an infinite simple group. 
If $(X,T)$ is a minimal subshift then $\lBrack T \rBrack^{\prime}$ is finitely 
generated. 
\end{thm}

It shall not concern us much exactly what a \emph{minimal subshift} 
is by definition, beyond the conclusion of Theorem \ref{FullGrpSimpleThm} and 
Theorem \ref{ManySubshiftsThm} below. 

\begin{thm}[Theorem 5.1 of \cite{GrigMedy}] \label{GrigMedyThm}
For any minimal Cantor dynamical system $(X,T)$, $\lBrack T \rBrack$ is LEF. 
\end{thm}

\begin{thm}[\cite{JusMon}] \label{JusMonThm}
For any minimal Cantor dynamical system $(X,T)$, $\lBrack T \rBrack$ is amenable. 
\end{thm}

Theorems \ref{GrigMedyThm} and \ref{JusMonThm} 
already imply that $\lBrack T \rBrack$ is always 
weakly locally stable, by Lemma \ref{amenwlslem}. 
Further, for $(X,T)$ a minimal subshift, 
our criterion for local stability from Theorem \ref{IRSMainThm} 
is applicable to $\lBrack T \rBrack^{\prime}$. 
The classification of IRSs of 
$\lBrack T \rBrack^{\prime}$ 
is provided by the next result. 

\begin{thm}[Corollary 1.4 of \cite{Zheng}] \label{ZhengThm}
Let $(X,T)$ be a minimal Cantor dynamical system. 
Let $\mu$ be an ergodic IRS of $\lBrack T \rBrack^{\prime}$. 
Then either: 
\begin{itemize}
\item[(i)] $\mu = \delta_{\lbrace e\rbrace}$ 
or $\delta_{\lBrack T \rBrack^{\prime}}$

\item[or (ii)] There exists $k \in \mathbb{N}$ and 
$T$-invariant 
ergodic probability measures $\nu_i$ on $X$ 
such that $\mu$ is the pushforward 
of $\nu_1 \times \cdots \times \nu_k$ 
under the map 
$\Stab : X^k \rightarrow \Sub\big( \lBrack T \rBrack^{\prime} \big)$, 
given by: 
\begin{equation*}
\Stab_k (x_1 ,\ldots ,x_k)=\bigcap_{i=1} ^k \Stab_{\lBrack T \rBrack^{\prime}} (x_i). 
\end{equation*}
That is, $\mu = \Stab_{\ast} (\nu_1 \times \cdots \times \nu_k)$, 
where $\lBrack T \rBrack^{\prime}$ acts diagonally 
on $X^k$. 
\end{itemize}
\end{thm}

Our proof of local stability for the 
groups $\lBrack T \rBrack^{\prime}$ (Theorem \ref{FullGrpMainThm} below)
will be based on the proof of Theorem \ref{GrigMedyThm} given in \cite{GrigMedy}: 
we show that the marked finite groups converging to 
(some marking of) $\lBrack T \rBrack^{\prime}$ in the space of marked groups, 
which are constructed in the proof of that Theorem, 
admit IRSs converging to the IRSs of $\lBrack T \rBrack^{\prime}$ 
described in Theorem \ref{ZhengThm}. 
The argument goes as follows: 
a small clopen set $B$ of $X$ determines a clopen partition 
$\Xi$ of $X$ (the Kakutani-Rokhlin partition), 
on which $\lBrack T \rBrack^{\prime}$ admits a partial action 
by permutations, generating a group $\Delta(\Xi) \leq \Sym(\Xi)$. 
If $\nu_i$ are probability measures on $X$ (as in Theorem \ref{ZhengThm} (ii)), 
then $\nu_i$ imparts a mass to each point in $\Xi$. 
Pushing forward under the stabilizer map, 
we obtain an IRS $\nu_{\Xi}$ of $\Delta(\Xi)$. 
Taking a nested sequence of clopen sets $B_n$ (intersecting in a point) 
the sequence of finite groups $\Delta(\Xi_n)$ and their IRSs $\nu_{\Xi}$ 
will satisfy the conditions of Theorem \ref{IRSMainThm}. 

\begin{defn}
Let $(X,T)$ be a minimal Cantor dynamical system. 
A \emph{$T$-tower} is a finite family: 
\begin{center}
$\xi = \lbrace B , TB , \ldots , T^{h-1}B \rbrace$, 
\end{center}
where $B \subseteq X$ is a nonempty clopen set 
such that the sets $B , TB , \ldots , T^{h-1}B$ 
are pairwise disjoint. 
We refer to the positive integer $h$ as 
the \emph{height} of the $T$-tower $\xi$. 
A \emph{Kakutani-Rokhlin (K-R) partition} of $X$ is a finite clopen partition $\Xi$ of $X$ 
which is a disjoint union of $T$-towers, 
that is, a clopen partition of the form: 
\begin{equation} \label{KRPartDescr}
\Xi = \lbrace T^i B_v :
 1 \leq v \leq q ; 0 \leq i \leq h_v - 1 \rbrace
\end{equation}
The sets $T^i B_v$ are the \emph{atoms} of the partition $\Xi$. 
The sets: 
\begin{equation*}
B(\Xi) = \coprod_{v=1} ^q B_v \text{ and } H(\Xi) = \coprod_{v=1} ^q T^{h_v - 1} B_v
\end{equation*}
are called respectively the \emph{base} and the \emph{roof} of $\Xi$. 
\end{defn}

\begin{rmrk}
\normalfont
We note some easy consequences of the definition. 
\begin{itemize}
\item[(i)] Any $T$-tower of height $h \geq 2$ 
can be written as the disjoint union of two $T$-towers 
of smaller height. 
Thus the set of $T$-towers making up a K-R 
partition $\Xi$ (and their heights) 
are not intrinsic to the partition $\Xi$ itself; 
rather we consider the division of the atoms 
into $T$-towers to be part of the data of $\Xi$. 

\item[(ii)] Since $T$ is injective, it maps $H(\Xi)$ into $B(\Xi)$ 
(any point of $X$ not lying in $B(\Xi)$ is in the image under $T$ of some 
atom of $\Xi$ disjoint from $H(\Xi)$). 
Applying the same reasoning to $T^{-1}$, 
we have that $TH(\Xi) = B(\Xi)$. 
\end{itemize}
\end{rmrk}

Given a K-R partition $\Xi$ of the form (\ref{KRPartDescr}), 
we shall write $h(\Xi) = \min_{1 \leq v \leq q} h_v$ 
to denote the minimal height among the $T$-towers 
appearing in $\Xi$. 
The next construction is described in 
\cite{GrigMedy} Remark 3.3. 

\begin{lem} \label{RefineKRPartLem}
Let $\Xi$ be a K-R partition of $X$ 
and let $\Pi$ be a finite clopen partition of $X$. 
Then there exists a K-R partition 
$\Xi'$ of $X$ which is a common refinement 
of $\Xi$ and $\Pi$, 
such that $B(\Xi)=B(\Xi')$, $H(\Xi)=H(\Xi')$ and 
$h(\Xi) = h(\Xi')$. 
\end{lem}

\begin{proof}
Let $\lbrace B , TB , \ldots , T^{h-1} B \rbrace$ 
be a $T$-tower of $\Xi$. 
For each $0 \leq i \leq h-1$, 
$\Pi_i = \lbrace B \cap T^{-i}P : P \in \Pi \rbrace \setminus \lbrace \emptyset \rbrace$ 
is a finite clopen partition of $B$. 
Let $\lbrace C_1 , \ldots C_n \rbrace$ be any finite 
clopen partition of $B$ which is a common refinement 
of $\Pi_0 , \Pi_1 , \ldots , \Pi_{h-1}$. 
Then for $1 \leq j \leq n$, 
$\xi_j =\lbrace C_j,T C_j ,\ldots ,T^{h-1}C_j\rbrace$ 
is a $T$-tower of height $h$; 
each $T^i C_j$ is contained in $T^i B$ 
and in a unique element of $\Pi$, and: 
\begin{equation*}
\bigcup_{j=1 } ^n \bigcup_{C \in \xi_j} C = 
\bigcup_{i=0} ^{h-1} T^i B. 
\end{equation*}
Applying this construction to each $T$-tower of $\Xi$ 
yields the desired K-R partition $\Xi'$. 
\end{proof}

Our next Proposition is a summary of the content of \cite{GrigMedy} Section 3. 

\begin{propn} \label{KRBasicProp}
Let $(X,T)$ be a minimal Cantor dynamical system. 
For any increasing sequence $(m_n)$ of positive integers, 
there exists a sequence of K-R partitions: 
\begin{equation} \label{KR1-5eqn}
\Xi_n =\lbrace T^i B_v ^{(n)} :0\leq i\leq h_v ^{(n)} -1;v=1,\ldots ,v_n \rbrace
\end{equation} 
of $X$ satisfying the following: 
\begin{itemize}
\item[(i)] The union of the $\Xi_n$ generates the topology on $X$; 

\item[(ii)] $\Xi_{n+1}$ refines $\Xi_n$; 

\item[(iii)] $B(\Xi_{n+1}) \subseteq B(\Xi_n)$ 
and there exists $x_0 \in X$ such that 
$\bigcap_n B (\Xi_n) = \lbrace x_0 \rbrace$; 

\item[(iv)] For all $n$, 
$h(\Xi_n) \geq 2 m_n +2$; 

\item[(v)] For all $n$ and $-m_n - 1 \leq i \leq m_n$, 
\begin{equation*}
\diam \big( T^i B(\Xi_n) \big) < 1/n. 
\end{equation*} 
\end{itemize}
\end{propn}

Following \cite{GrigMedy} Section 4, 
given a sequence $(\Xi_n)$ of K-R partitions of $X$ 
as in (\ref{KR1-5eqn}), 
satisfying (i)-(v) of Proposition \ref{KRBasicProp}, 
we say that an element $\pi \in \lBrack T \rBrack$ 
is an \emph{$n$-permutation} if the orbit cocycle 
$f_{\pi}$ is constant on each part of $\Xi_n$ 
and for all $1 \leq v \leq v_n$ 
and $0 \leq i \leq h_v ^{(n)} - 1$, 
$f_{\pi}$ satisfies $-i \leq f_{\pi} (x) \leq h_v ^{(n)}-i-1$ 
for all $x \in T^i B_v ^{(n)}$. 
Thus $\pi$ preserves each $T$-tower $\xi$ of $\Xi_n$, 
and induces a well-defined permutation on 
the set of atoms of $\xi$. 

\begin{rmrk} \label{n-PermRmrk}
The set of all $n$-permutations in $\lBrack T \rBrack$ 
forms a subgroup of $\lBrack T \rBrack$. 
This subgroup is isomorphic to 
$\Sym(h_1 ^{(n)}) \times \ldots \times \Sym(h_{v_n} ^{(n)})$, since any tuple of permutations 
of the atoms in each $T$-tower of $\Xi_n$ 
may be realized by some $n$-permutation 
in $\lBrack T \rBrack$. 
\end{rmrk}

There is also defined in \cite{GrigMedy} Section 4 
the notion of an \emph{$n$-rotation}, 
and we refer the reader there for the precise 
definition; 
the only fact that we require about $n$-rotations is 
the following, which is immediate from the 
definition. 

\begin{rmrk} \label{n-RotRmrk}
If $\rho \in \lBrack T \rBrack$ is an $n$-rotation, 
and $x \in X$ is such that $\rho(x) \neq x$, 
then $\lvert f_{\rho} (x) \rvert \geq \min (h_1 ,h_2)$, where $h_1$ (respectively $h_2$) is the height of 
the $T$-tower of $\Xi_n$ containing $x$ 
(respectively $\rho(x)$). 
\end{rmrk}

Henceforth $\Sym(\Xi_n)$ denotes the group 
of all permutations of the finite set $\Xi_n$. 
Our next Theorem shows how, for large $n$, an element of 
$\lBrack T \rBrack$ induces a well-defined $n$-permutation 
of $\Xi_n$, which in turn induces 
an element of $\Sym(\Xi_n)$. 
Note however that not every element of $\Sym(\Xi_n)$ 
need arise this way; see Remark \ref{n-PermRmrk} above. 

\begin{thm} \label{GrigMedyMainThm}
Let $(X,T)$ be a minimal Cantor dynamical system, 
and let $(\Xi_n)$ be a sequence of K-R 
partitions satisfying conditions (i)-(v) of 
Proposition \ref{KRBasicProp} above, 
for some increasing sequence of integers $(m_n)$. 
\begin{itemize}
\item[(i)] Let $g \in \lBrack T \rBrack$. 
For all $n$ sufficiently large, there exist unique 
$\pi_n (g) , \rho_n (g) \in \lBrack T \rBrack$ 
such that $g = \pi_n (g) \rho_n (g)$; $\pi_n (g)$ is an $n$-permutation 
and $\rho_n (g)$ is an $n$-rotation. 
\item[(ii)] For any $A \subseteq \lBrack T \rBrack$ finite, 
if $n$ is sufficiently large, then there is a local embedding 
$\phi_n : A \rightarrow \Sym(\Xi_n)$, given by 
$\phi_n (g)(T^i B_v ^{(n)}) =\pi_n (g)(T^i B_v ^{(n)})$. 
In particular, $\lBrack T \rBrack$ is LEF. 
\end{itemize}
\end{thm}

\begin{proof}
Item (i) is immediate from Theorem 4.7 of \cite{GrigMedy}. 
Item (ii) is proved as Theorem 5.1 of \cite{GrigMedy} 
(note that the statement of that Theorem does not specify the local embedding, 
but the local embedding given in the proof is precisely as we have described it). 
\end{proof}

Henceforth we assume that $(X,T)$ is a minimal subshift, 
so that Theorem \ref{FullGrpSimpleThm} applies. 
Fix a finite symmetric generating set $S$ for $\lBrack T \rBrack^{\prime}$, 
let $n \in \mathbb{N}$ and consider the ball 
$B_S (n) \subseteq \lBrack T \rBrack^{\prime}$. 
Recall that, for $g \in \lBrack T \rBrack$, 
$f_g : X \rightarrow \mathbb{Z}$ 
is the orbit cocycle of $g$. 

\begin{propn} \label{KR2ndProp}
Let $(m_n)$ be an increasing sequence 
of positive integers. 
There is a sequence $(\Xi_n)$ of K-R 
partitions of $X$ satisfying items (i)-(v) 
of Proposition \ref{KRBasicProp}, 
and additionally satisfying the following, 
for all $n \in \mathbb{N}$: 
\begin{itemize}
\item[(vi)] For all $g \in B_S (n)$, 
$h(\Xi_n) \geq 2 \max \lbrace \lvert f_g(x) \rvert : x \in X \rbrace + 2$; 

\item[(vii)] For all $g \in B_S (n)$, 
$f_g$ is constant on each part of $\Xi_n$; 

\item[(viii)] For all $g \in B_S (n)$, 
there exist unique 
$\pi_n (g) , \rho_n (g) \in \lBrack T \rBrack$ 
such that $g = \pi_n (g) \rho_n (g)$; $\pi_n (g)$ is an $n$-permutation 
and $\rho_n (g)$ is an $n$-rotation. 
Moreover the map 
$\phi_n : B_S (n) \rightarrow \Sym(\Xi_n)$, given by 
$\phi_n (g)(T^i B_v ^{(n)}) =\pi_n (g)(T^i B_v ^{(n)})$, 
is a local embedding. 
\end{itemize}
\end{propn}

\begin{proof} 
First, since $X$ is compact and 
the orbit cocycle is continuous, it is bounded. 
Therefore (replacing $(m_n)$ with a faster growing 
sequence if required) 
we can assume that 
for all $g \in B_S (n)$, 
\begin{center}
$m_n \geq \max \lbrace \lvert f_g(x) \rvert : x \in X \rbrace$, 
\end{center}
and then apply Proposition \ref{KRBasicProp} (iv). 

Second, we inductively 
refine each $\Xi_n$ such that 
properties (i)-(vi) still hold, 
and for all $g \in B_S (n)$, the orbit cocycle 
$f_g$ is constant on each part of $\Xi_n$. 
Supposing that we have already refined $\Xi_{n-1}$, 
for each $g \in B_S (n)$ let $\mathcal{C}_g$ 
be a finite clopen partition of $X$ such that 
$f_g$ is constant on each part of $\mathcal{C}_g$. 
Applying Lemma \ref{RefineKRPartLem} repeatedly, 
we replace $\Xi_n$ 
with a finer finite clopen partition, 
which is also a refinement of both $\Xi_{n-1}$ 
and all $\mathcal{C}_g$. 
This process clearly preserves properties (i)-(vi) 
(property (iv) holding by the final part of 
Lemma \ref{RefineKRPartLem}). 

Finally, passing to a subsequence 
of $(\Xi_n)$, and applying 
Theorem \ref{GrigMedyMainThm} (i) 
to elements $g \in B_S (n)$, 
we may assume that the decomposition 
$g = \pi_n (g) \rho_n (g)$ exists and is unique. 
Moreover by Theorem \ref{GrigMedyMainThm} (ii), 
applied to the finite subsets $A = B_S (n)$, 
and again passing to a subsequence of $(\Xi_n)$, 
we may assume that the given map 
$\phi_n : B_S (n) \rightarrow \Sym(\Xi_n)$ 
is a well-defined local embedding. 
Note that passing to a subsequence of $(\Xi_n)$
preserves properties (i)-(vii). 
\end{proof}

Henceforth we fix a sequence $(\Xi_n)$ of K-R partitions 
of $X$ satisfying properties (i)-(viii) 
of Propositions and \ref{KRBasicProp} and \ref{KR2ndProp}, 
with respect to some increasing sequence $(m_n)$. 
Let $\phi_n : B_S (n) \rightarrow \Sym(\Xi_n)$ 
be the local embedding as in 
Proposition \ref{KR2ndProp} (viii). 
For $x \in X$, we write $[x]_n \in \Xi_n$ for 
the (unique) atom of $\Xi_n$ containing $x$. 

\begin{lem} \label{BlockStabLem}
For all $n \in \mathbb{N}$, for all $g \in B_S (n)$ 
and all $x \in X$, the following are equivalent. 
\begin{itemize}
\item[(i)] $g (x) = x$; 

\item[(ii)] For all $y \in [x]_n$, $g(y)=y$; 

\item[(iii)] $\phi_n (g)([x]_n) = [x]_n$. 

\end{itemize}
\end{lem}

\begin{proof}
If $g(x)=x$, then by minimality of $T$ on $X$, 
$f_g (x) = 0$. By Proposition \ref{KR2ndProp} (vii), 
for all $y \in B$, $f_g (y) = 0$. 
Thus (i) and (ii) are equivalent. 

Write $g = \pi_n (g) \rho_n (g)$ 
as in Theorem \ref{GrigMedyMainThm}, 
with $\pi_n (g) \in \lBrack T \rBrack$ 
an $n$-permutation 
and $\rho_n (g) \in \lBrack T \rBrack$ an $n$-rotation, 
so that $\phi_n(g) ([x]_n)=\pi_n (g) ([x]_n)$. 
As in Remark \ref{n-PermRmrk}, 
$\pi_n (g)^{-1}$ is an $n$-permutation also. 
Suppose that (ii) holds, 
so that $\pi_n (g)^{-1} ([x]_n) = \rho_n (g)([x]_n)$. 
Since an $n$-permutation preserves each $T$-tower 
of $\Xi_n$, and sends atoms to atoms, 
$\rho_n (g)([x]_n)$ is an atom of $\Xi_n$ 
in the same $T$-tower as $[x]_n$. 
Let the height of this tower by $h$. 
If (iii) fails, so that $\rho_n (g)([x]_n)\neq [x]_n$, 
then by Remark \ref{n-RotRmrk}, 
for $y \in [x]_n$, 
$\lvert f_{\rho_n (g)} (y) \rvert \geq h$. 
By the cocycle relation 
(see Remark \ref{OrbitCocycRmrk} (ii)), 
\begin{align*}
0 = f_g (y) = f_{\rho_n (g)} (y) + f_{\pi_n (g)} \big( \rho_n (g)(y) \big)
\end{align*}
so $\lvert f_{\pi_n (g)} ( \rho_n (g)(y) )\rvert\geq h$ also, contradicting the definition of 
an $n$-permutation. 

Conversely suppose (iii) holds and let $y \in [x]_n$. 
Then $\pi_n(g)$ preserves $[x]_n$, 
so (by the bound on $f_{\pi_n(g)}$ 
from the definition of an $n$-permutation) 
$\pi_n(g)$ fixes $[x]_n$ pointwise, 
hence so does $\pi_n(g)^{-1}$. 
We have $g^{-1} = \rho_n(g)^{-1} \pi_n(g)^{-1}$, 
so $g^{-1} (y) = \rho_n(g)^{-1} (y)$, hence: 
\begin{equation} \label{CocycEqnRot}
f_{g^{-1}} (y) = f_{\rho_n(g)^{-1}} (y) 
= - f_{\rho_n(g)} \big( \rho_n(g)(y) \big). 
\end{equation}
(by the cocycle relation). 
By Proposition \ref{KR2ndProp} (vi), 
the left-hand side of (\ref{CocycEqnRot}) 
has absolute value less than $h(\Xi_n)$. 
Applying Remark \ref{n-RotRmrk} to the right-hand side, 
we have $f_{g^{-1}} (y) = 0$, so $g (y)=y$. 
\end{proof}

Recall that $S$ is a finite generating set for 
$\lBrack T \rBrack'$. 
Let $S = \lbrace s_1 , \ldots , s_d \rbrace$, 
so that $\mathbf{S} = (s_1 , \ldots , s_d)$ 
is a $d$-marking on $\lBrack T \rBrack'$. 

\begin{propn} \label{StabIRSPartCosof}
Let $\nu_1 , \ldots , \nu_k$ 
be $T$-invariant ergodic probability measures on $X$, 
and let: 
\begin{center}
$\mu = (\Stab_k)_{\ast} (\nu_1 \times \ldots \times \nu_k) \in \IRS(\lBrack T \rBrack',\mathbf{S})$ 
\end{center}
be as in Theorem \ref{ZhengThm} (ii). 
Then $\mu$ is partially cosofic. 
\end{propn}

\begin{proof}
Let $\phi_n : B_S (n) \rightarrow \Sym(\Xi_n)$ 
be the local embedding 
as in Proposition \ref{KR2ndProp}; 
let $\Delta_n = \langle \phi_n (S) \rangle \leq \Sym(\Xi_n)$, and let 
$\mathbf{T}_n = (\phi_n(s_1) , \ldots , \phi_n(s_d))$, 
a $d$-marking on $\Delta_n$. 
Then the sequence $(\Delta_n,\mathbf{T}_n)$ 
converges to $(\lBrack T \rBrack',\mathbf{S})$ in 
$\mathcal{G}_d$, by Lemma \ref{MarkedGrpsLEFLemma}. 

For each $n \in \mathbb{N}$ and $1 \leq j \leq k$, 
there are induced probability measures 
$\overline{\nu}_j ^{(n)}$ on the finite discrete set $\Xi_n$, 
given by $\overline{\nu}_j ^{(n)} (\lbrace T^i B_v ^{(n)} \rbrace)= \nu_j (T^i B_v ^{(n)}) 
= \nu_j (B_v ^{(n)})$ (the second equality holding by $T$-invariance of $\nu_j$). 
Since each $\phi_n (s_m)$ preserves each $T$-tower of 
$\Xi_n$, 
\begin{center}
$\overline{\nu}_j ^{(n)} \big(\phi_n (s_m)\lbrace T^i B_v ^{(n)} \rbrace \big)
 = \overline{\nu}_j ^{(n)} \big( \lbrace T^i B_v ^{(n)}\rbrace \big)$ 
\end{center}
so $\overline{\nu}_j ^{(n)}$ is $\Delta_n$-invariant, 
hence 
$(\overline{\nu}_1 ^{(n)} \times \cdots \times \overline{\nu}_k ^{(n)})$ is 
a $\Delta_n$-invariant probability measure 
on $\Xi_n ^k$ (with $\Delta_n$ acting diagonally). 
Thus $\mu_n = \Stab _{\ast} (\overline{\nu}_1 ^{(n)} \times \cdots \times \overline{\nu}_k ^{(n)}) \in 
\IRS (\Delta_n,\mathbf{T}_n)$. 
Let $\mathcal{B}_n$ be the family of clopen 
subsets of $X^k$ which are unions of 
sets of the form $B_1 \times \ldots \times B_k$, 
for $B_i \in \Xi_n$. 
Then $\mathcal{B}_n$ is a (finite) $\sigma$-algebra; 
indeed there is a (unique) isomorphism of $\sigma$-algebras 
$\Psi_n : \mathcal{P}(\Xi_n ^k) \rightarrow \mathcal{B}_n$ 
extending $\Psi \big( \lbrace (B_1,\ldots,B_k) \rbrace \big) = B_1 \times \ldots \times B_k$. 
Moreover, since: 
\begin{align*}
(\overline{\nu}_1 ^{(n)} \times \cdots \times \overline{\nu}_k ^{(n)})\big( \lbrace (B_1,\ldots,B_k) \rbrace \big) 
& = \nu_1 (B_1) \cdots \nu_k (B_k) \\
& = (\nu_1 \times \ldots \times \nu_k) (B_1 \times \ldots \times B_k)
\end{align*}
for any $B_1 , \ldots , B_k \in \Xi_n$, 
we have: 
\begin{equation} \label{PsiPresMeasEqu}
(\nu_1 \times \ldots \times \nu_k)\big( \Psi (A )\big) = (\overline{\nu}_1 ^{(n)} \times \cdots \times \overline{\nu}_k ^{(n)})(A)
\end{equation}
for all $A \subseteq \Xi_n ^k$. 

We shall use the criterion (\ref{W*ConvCrit}) 
from Section \ref{IRSSect} to show that 
the sequence $(\mu_n)$ converges to $\mu$ 
in $\IRS(\mathbb{F},\mathbf{X})$. 
To this end, let $r \in \mathbb{N}$ and 
$W \subseteq B_X (r)$. 
We show that $\mu_n (C_{r,W}) = \mu (C_{r,W})$ 
for all $n \geq r$. 
By definition of the pushforward measures, 
$\mu (C_{r,W})$ is the probability that, 
for $\nu_i$-random points $x_1 , \ldots , x_k \in X$, 
every $w \in W$ satisfies $\pi_{\mathbf{S}} (w)(x_i)=x_i$ for all $1 \leq i \leq k$, 
but for every $w \in B_X (r) \setminus W$, 
there exists $1 \leq i \leq k$ such that 
$\pi_{\mathbf{S}} (w)(x_i) \neq x_i$. 
Similarly, $\mu_n (C_{r,W})$ 
is the probability that, for $\overline{\nu}_i ^{(n)}$-random points $B_1 , \ldots , B_k \in \Xi_n$, 
every $w \in W$ satisfies $\pi_{\mathbf{T}_n} (w)(B_i)=B_i$ for all $1 \leq i \leq k$, 
but for every $w \in B_X (r) \setminus W$, 
there exists $1 \leq i \leq k$ such that 
$\pi_{\mathbf{T}_n} (w)(B_i) \neq B_i$. 
Our claim is that these probabilities are equal. 

By Lemma \ref{BlockStabLem}, 
for each $w \in B_X (r)$ the set of points 
$x \in X$ for which $\pi_{\mathbf{S}}(w)(x)=x$ 
is a union of whole atoms of $\Xi_n$. 
Moreover, since 
$(\phi_n \circ\pi_{\mathbf{S}})(w)=\pi_{\mathbf{T}_n}(w)$, 
Lemma \ref{BlockStabLem} yields that, 
for $x \in X$, 
\begin{equation*}
\lbrace B \in \Xi_n : \pi_{\mathbf{T}_n}(w)(B)=B \rbrace
= \lbrace [x]_n : x \in X, \pi_{\mathbf{S}}(w)(x)=x \rbrace. 
\end{equation*}
It follows that the set: 
\begin{equation*}
\lbrace (x_1,\ldots,x_k) \in X^k : \forall w \in B_X (r), w \in W \Leftrightarrow \forall i , \pi_{\mathbf{S}}(w)(x_i)=x_i \rbrace
\end{equation*} 
belongs to the $\sigma$-algebra $\mathcal{B}_n$, 
and is precisely $\Psi_n(A_{r,W})$, where: 
\begin{equation*}
A_{r,W} = \lbrace (B_1,\ldots,B_k) \in \Xi_n ^k : 
\forall w \in B_X (r), w \in W \Leftrightarrow \forall i , \pi_{\mathbf{T}_n}(w)(B_i)=B_i \rbrace. 
\end{equation*}
By (\ref{PsiPresMeasEqu}) and the discussion 
immediately following, 
\begin{equation*}
\mu (C_{r,W}) 
= (\nu_1 \times \ldots \times \nu_k)\big( \Psi (A_{r,W} )\big) = (\overline{\nu}_1 ^{(n)} \times \cdots \times \overline{\nu}_k ^{(n)})(A_{r,W})
= \mu_n (C_{r,W}) 
\end{equation*}
as desired. 
\end{proof}

\begin{thm} \label{FullGrpMainThm}
Let $(X,T)$ be a Cantor minimal subshift. 
Then $\lBrack T \rBrack'$ is locally stable. 
\end{thm}

\begin{proof}
By Theorems \ref{FullGrpSimpleThm} and \ref{JusMonThm}, 
the criterion for local stability 
from Corollary \ref{IRSMainCoroll} (iii) applies. 
Let $\mu \in \IRS (\lBrack T \rBrack',\mathbf{S})$ 
be an ergodic IRS. 
If $\mu$ is as in Theorem \ref{ZhengThm} (i), 
then $\mu$ is partially cosofic by 
Theorem \ref{GrigMedyThm} 
and Remark \ref{TrivialIRSRmrk}. 
If $\mu$ is as in Theorem \ref{ZhengThm} (ii), 
the partial cosoficity is precisely the content 
of Proposition \ref{StabIRSPartCosof}. 
\end{proof}

\begin{thm} \label{ManySubshiftsThm}
There is a continuum of pairwise nonisomorphic 
groups of the form $\lBrack T \rBrack'$, 
for $(X,T)$ a minimal subshift. 
\end{thm}

\begin{proof}
This is proved, for example in \cite{Mat}[p.246]. 
Alternatively, an explicit continuous family of 
minimal subshifts 
$\lbrace (X_r,T_r) \rbrace_{r \in [2,\infty)}$ 
is constructed in Section 5 of \cite{BradDona}. 
It is shown that for $2 \leq r < r'$, 
the isomorphism-types of $\lBrack T_r \rBrack'$ 
and $\lBrack T_{r'} \rBrack'$ are distinguished 
by their \emph{LEF growth functions} 
(see the statements of Theorems 1.5 and 1.6 
from  \cite{BradDona}).
\end{proof}

\begin{thm} \label{MainThm}
There is a continuum of pairwise nonisomorphic 
finitely generated groups, 
which are locally stable but not weakly stable. 
\end{thm}

\begin{proof}
Let $(X,T)$ be a Cantor minimal subshift. 
By Theorem \ref{FullGrpMainThm} $\lBrack T \rBrack'$ is locally stable. 
By Theorems \ref{FullGrpSimpleThm} and \ref{JusMonThm}, 
$\lBrack T \rBrack'$ is finitely generated 
amenable but not residually finite. 
By Theorem \ref{ArzPauThm} we conclude 
that $\lBrack T \rBrack'$ is not weakly stable. 
The result then follows 
from Theorem \ref{ManySubshiftsThm}. 
\end{proof}

\section{Concluding remarks}

There are many sources of examples of amenable LEF groups, 
and it is interesting to ask which of these may be locally stable. 
For instance the (regular restricted) wreath product of any two 
amenable LEF groups is amenable LEF. 
It follows immediately from Lemma \ref{amenwlslem} 
that the class of finitely generated amenable 
weakly locally stable groups is closed 
under wreath products. 
We may therefore ask if the same holds 
when we strengthen ``weakly locally stable'' 
to ``locally stable''. 

\begin{qu} \label{WPBigQu}
Let $\Gamma$ and $\Delta$ be finitely generated 
amenable locally stable groups. 
Must $\Delta \wr \Gamma$ be locally stable? 
\end{qu}

As a modest first step, one may ask for the following. 

\begin{conj} \label{WPConj}
Let $\Delta$ be a finite group. 
Then $\Delta \wr \mathbb{Z}$ is locally stable
\end{conj}

Note that by \cite{LevLubWP}, 
the wreath product of any two finitely generated abelian groups 
is stable. 
In particular, the special case of Conjecture \ref{WPConj} for which $\Delta$ 
is abelian is known to hold. 
By contrast, if $\Delta$ is nonabelian, 
then $\Delta \wr \mathbb{Z}$ is amenable but not residually finite, 
so not even weakly stable. 

In a similar vein, one may ask for a generalization of Theorem \ref{AltEnrLocStabThm}. 
For any group $\Gamma$, the group $\FAlt(\Gamma)$ of finitely supported 
even permutations of the \emph{set} $\Gamma$ 
is normalized by the image in $\Sym(\Gamma)$ of the regular representation 
of $\Gamma$. We therefore obtain a semidirect product 
$\mathcal{A}(\Gamma) = \FAlt(\Gamma) \rtimes \Gamma$, 
the \emph{alternating enrichment} of $\Gamma$. 
If $\Gamma$ is respectively finitely generated, amenable or LEF, 
then so is $\mathcal{A} (\Gamma)$. 
On the other hand, if $\Gamma$ is infinite, 
then $\mathcal{A} (\Gamma)$ is not residually finite. 

\begin{qu}
Under what conditions on $\Gamma$ is $\mathcal{A}(\Gamma)$ locally stable? 
\end{qu}

Next, one may ask for other applications of 
Proposition \ref{liftingalmosthomsprop}. 
Many famous examples of ``monster'' groups are constructed 
as limits of sequences of marked epimorphisms of $d$-marked groups. 
For instance, Tarski monsters and free Burnside groups both arise 
in this way, as limits of sequences of marked epimorphisms 
of finitely presented groups satisfying a ``small-cancellation'' condition 
(see for instance \cite{Olsh}). 
It remains a well-known open problem whether or not such groups are LEF 
(see for instance Problem 5.14 (ii) of \cite{Capr}: 
if it were even the case that such monster groups were not 
limits in $\mathcal{G}_d$ of nonabelian finite simple groups, 
then dramatic consequences would follow). 
It is therefore also interesting to ask whether such groups are locally stable, 
which would follow from a positive answer to our next question. 

\begin{qu}
Is every finitely presented small-cancellation group stable? 
\end{qu}

One may further ask to what extent Proposition \ref{liftingalmosthomsprop} may 
be generalized. Although Remark \ref{LimitQuotRmrk} shows that 
local stability is not a closed property in $\mathcal{G}_d$, 
it is reasonable to see the existence of a sequence of marked 
(locally) stable groups converging in $\mathcal{G}_d$ 
to (a $d$-marking of) $\Gamma$ 
as evidence that $\Gamma$ is locally stable, 
especially if the groups in the sequence are in some sense ``uniformly'' stable. 
For instance, finite rank free groups are surely ``at least as stable'' 
as any other finitely generated groups. 
The closure in $\mathcal{G}_d$ of the set of ($d$-markings of) free groups 
of rank $\leq d$ is precisely the set of $d$-generated \emph{limit groups}. 
Beyond free groups and free abelian groups, 
the only limit groups for which any stability results are known are 
the fundamental groups of closed oriented surfaces 
(as described in \cite{LazLevMin}; see below). 
Moreover, every limit group is finitely presented, 
so local stability can be upgraded automatically to stability. 

\begin{conj}
Every finitely generated limit group is permutation stable. 
\end{conj}

There is a generalization of permutation stability, 
called \emph{flexible stability}, 
under which we may slightly enlarge the finite domains on which the images of 
our almost-homomorphisms act before seeking asymptotically equivalent actions 
on those domains. 
For example, fundamental groups of closed oriented surfaces 
are flexibly stable \cite{LazLevMin}, but the question of their stability remains open. 
In another direction, it is known that the groups $\SL_d (\mathbb{Z})$ ($d \geq 3$) 
are not stable \cite{BecLubT}, but unknown whether they are flexibly stable. 
By \cite{BowBur}, if $\SL_d (\mathbb{Z})$ is flexibly stable for some $d \geq 5$, 
then there exists a nonsofic group. 
Just as local stability generalizes stability, 
one may analogously define a notion of flexible local stability, 
generalizing flexible stability. 
A flexibly locally stable sofic group is still necessarily LEF. 

\begin{prob} 
Find examples of groups which are flexibly locally stable but not locally stable. 
\end{prob}

In a related direction, 
the following question was posed by A. Lubotzky. 

\begin{qu} \label{PropTQu}
Does there exist a finitely generated locally stable group with Kazhdan's Property $(T)$? 
\end{qu}

The corresponding question for stable groups 
has a negative answer \cite{BecLubT}; 
in particular, no example satisfying 
Question \ref{PropTQu} may be found among finitely presented groups. 
Note that all of the locally stable but non-stable 
groups we have constructed are amenable, 
hence far from being Property $(T)$ groups. 



Finally, there are many other versions of 
``stability of metric approximations for groups'' 
besides stability in permutations. 
For any reasonable family of compact groups 
equipped with bi-invariant metrics $d$, 
one can define ``almost homomorphism'' in just 
the same way as for the Hamming metrics 
on the finite symmetric groups. 
Corresponding to the definition of a sofic group, 
there is a notion of ``$d$-approximable'' group, 
and a notion of a ``$d$-stable'' group 
corresponding to a group which is stable in 
permutations. 
We have alluded to \emph{Frobenius approximable} 
and \emph{Frobenius stable} groups in the 
Introduction, but one may also consider, 
for example, the Hilbert-Schmidt metrics on 
unitary groups (the groups which are 
HS-approximable are precisely the \emph{hyperlinear} 
groups); the rank metrics on groups 
of invertible matrices over fields 
(leading to the class of \emph{linear sofic} groups), 
or all finite groups equipped with bi-invariant 
metrics, leading to the \emph{weakly sofic} groups 
(see \cite{ArzChe} and the references therein; 
see \cite{ArPa} for a characterization 
in terms of liftings of homomorphisms to 
metric ultraproducts). 
In all these cases, 
finitely generated groups which are $d$-approximable 
and $d$-stable must be residually finite. 
One may similarly conceive of a notion of 
$d$-local stability, 
such that $d$-approximable 
and $d$-locally stable implies LEF. 

\begin{prob} \label{MetStabProb}
For each type of metric approximation discussed above, 
study the class of ``$d$-locally stable'' groups. 
Produce examples of groups which are 
$d$-locally stable but not $d$-stable. 
\end{prob}

Since circulation of a preliminary version of the present work, the last three Problems have been addressed 
in \cite{F3GerSpa}. 
Therein, the following are achieved: 
\begin{itemize}
\item[(i)] A general framework for metric local stability 
is described, which incorporates a notion of 
flexible local stability; 

\item[(ii)] Lubotzky's Question \ref{PropTQu} 
is answered in the negative; 

\item[(iii)] Problem \ref{MetStabProb} is studied 
in particular for the Hilbert-Schmidt distance 
on the unitary groups $U(n)$. 
\end{itemize}

\subsection*{Acknowledgements}

I am grateful to Goulnara Arzhantseva; Oren Becker; Francesco Fournier-Facio 
and Alex Lubotzky, whose helpful comments and questions on 
a preliminary version of this work have helped to inform my 
thinking on this subject. 
Finally, I am glad to thank the anonymous referee, 
whose corrections have improved the exposition of this paper greatly, 
and whose suggestions were so intelligently formulated 
as to make their report on my paper an unusual pleasure to read.

\end{document}